\documentclass[reqno,11pt]{amsart}

\newcommand\version{July 17, 2026}

\usepackage{esint}
\usepackage{amsthm}
\usepackage{amssymb}
\usepackage{graphicx}			
\usepackage{appendix}

\usepackage{paralist}
\usepackage{wrapfig}
\usepackage[draft]{pdfcomment}
\usepackage{bbm}
\usepackage{fontenc}
\usepackage{inputenc}
\usepackage[arrow, matrix, curve]{xy}
\usepackage{color}

\usepackage{tikz-cd}

\usepackage{enumitem}

\usepackage{mathtools}

\newtheorem{thm}{Theorem}[section]

\newtheorem{lem}[thm]{Lemma}

\newtheorem{cor}[thm]{Corollary}
\newtheorem{prp}[thm]{Proposition}

\theoremstyle{definition}

\theoremstyle{remark}
\newtheorem{rem}[thm]{Remark}

\newtheorem{remsthm}[thm]{Remarks on Theorem}

\newtheorem{remsprp}[thm]{Remarks on Proposition}

\newcommand{\R}{\mathbb{R}}

\newcommand{\N}{\mathbb{N}}

\newcommand{\HH}{\mathbb{H}}

\newcommand{\Lo}{\mathbb{L}}

	\usepackage[nodayofweek, level]{datetime}

\title[Quantitative Lorentzian isoperimetric inequalities]{Quantitative Lorentzian isoperimetric inequalities \\in conical Minkowski spacetimes }

\author[Christian Lange]{Christian Lange}
\address[Christian Lange]{Ludwig-Maximilians-Universit\"at M\"unchen, Mathematisches Institut\newline\indent Theresienstr. 39, 80333 München, Germany}
\email{lange@math.lmu.de}

\author[Jonas W. Peteranderl]{Jonas W. Peteranderl}
\address[Jonas W. Peteranderl]{Ludwig-Maximilians-Universit\"at M\"unchen, Mathematisches Institut\newline\indent Theresienstr. 39, 80333 München, Germany}
\email{peterand@math.lmu.de}


\date{\version}
\thanks{\copyright\, 2026 by the authors. This paper may be reproduced, in its entirety, for non-commercial purposes.}

\begin{document}

\begin{abstract}
We establish optimal stability estimates in terms of the Fraenkel asymmetry with universal dimensional constants for a Lorentzian isoperimetric inequality due to Bahn and Ehrlich and, as a consequence, for a special version of a Lorentzian isoperimetric inequality due to Cavalletti and Mondino. For the Bahn--Ehrlich inequality the Fraenkel asymmetry enters the stability result quadratically like in the Euclidean case while for the Cavalletti--Mondino inequality the Fraenkel asymmetry enters linearly. As it turns out, refining the latter inequality through an additional geometric term allows us to recover the more common quadratic stability behavior. Along the way, we provide simple, self-contained proofs for the above isoperimetric-type inequalities. Moreover, in a fixed conical Minkowski spacetime, we use a Lipschitz bound, naturally provided by the causal structure, to upgrade our quantitative control to a Hausdorff stability estimate. 
This estimate is formulated in terms of a distance defined by Bahn and Ehrlich, which restricts to a natural Hausdorff-type metric on the space of Cauchy hypersurfaces. 
\end{abstract}

\maketitle
\section{Introduction and main results}

Stability analysis investigates how close an object that almost attains the optimum is to an actual optimizer if one exists. Over the past two decades, this field has gained much attention and is nowadays an integral part in the study of functional and geometric inequalities; see \cite{Fi13,Fr24,Fu15}, for instance. Among the latter, the \textit{isoperimetric inequality} has been of primary interest in this context; see e.g.~\cite{Bo24,Fu89,Ha92,FMP08}. This inequality resolves the isoperimetric problem, which can be traced back to the legend of the Carthaginian queen Dido and asks for the maximal volume that can be enclosed by a given area; see e.g.~\cite{Os78,Ch01}.

Early instances of stability results in this direction are inequalities due to Bernstein \cite{Be05} and Bonnesen \cite{Bo24}. The latter implies that an almost maximizer of the isoperimetric inequality in the Euclidean plane is Hausdorff-close to a disk in a  quantitative sense. In higher dimensions there is no hope for such a Hausdorff-type stability without further assumptions, as arbitrarily long and suitably thin ``tentacles'' can be attached to the maximizer of the Euclidean isoperimetric inequality, the round ball, without deviating from equality in the inequality much. An optimal control measured in terms of the volume deviation from balls was first obtained by Fusco, Maggi, and Pratelli in their seminal work \cite{FMP08} via a symmetrization argument. An alternative proof via optimal transport that also applies to
anisotropic settings was found by Figalli, Maggi, and Pratelli shortly afterwards \cite{FMP10}; see also \cite{CL12} for yet another versatile proof technique. Figalli and Indrei \cite{FI13} used the approach in \cite{FMP10} to establish a quantitative version of a relative isoperimetric inequality inside an open convex cone $\mathcal{C} \subset \R^n$, $n\geq 2$, containing no lines due to Lions and Pacella \cite{LP90}. More precisely, 
the result in \cite{FI13} says that there exists a constant $c(n,\mathcal{C})>0$ depending on the dimension $n$ and on the cone $\mathcal{C}$ such that if $E\subset \mathcal{C}$ is a Borel set of finite De Giorgi perimeter relative to $\mathcal{C}$  and finite positive volume, then
\[ \left(\frac{\mathrm{Vol}(E\Delta (B_r^{Eucl}(0) \cap \mathcal{C}))}{\mathrm{Vol}(E)}\right)^2   
        \leq c(n,\mathcal{C}) \cdot \left(\frac{\mathrm{Per}(E|C)}{ n \mathrm{Vol}(B_1^{Eucl}(0) \cap \mathcal{C})^{\frac 1n} \mathrm{Vol}(E)^{\frac{n-1}{n}}}-1 \right),
\]
where $r>0$ is such that $\mathrm{Vol}(B_r^{Eucl}(0) \cap \mathcal{C})= \mathrm{Vol}(E)$ and $B_r^{Eucl}(x)$ denotes the $n$-dimensional Euclidean ball of radius $r>0$ and center $x\in \R^n$. Here the left side is the square of the volume deviation
known as {\em Fraenkel asymmetry} and the term
on the right side, which is always nonnegative by \cite{LP90}, is called the {\em (relative) isoperimetric deficit}. Such a quantitative isoperimetric estimate was, for instance, applied in the context of Gamow's liquid drop model \cite{Ga30} from nuclear physics to prove existence of ground states; see \cite{KM13,KM14}.

Isoperimetric-type bounds are 
relevant in mathematical general relativity as well; see e.g.~\cite{BC04,Hu06,BE13, EM13, CESY21}. However, in Lorentzian signature one can neither estimate from above the volume of spacetime regions \cite[p.~3]{CM25} nor the area of spacelike hypersurfaces with fixed boundary in terms of their boundary area; see \cite{TW22, Le23} though. In fact, minimization problems in Riemannian geometry rather correspond to maximization problems in Lorentzian geometry and vice versa. For instance, geodesics locally minimize arclength in the Riemannian case while they locally maximize it in the Lorentzian case. This reverse behavior is fundamental to special relativity and is reflected in phenomena such as the \textit{twin paradox}. The main feature of Minkowski space that underlies this paradox is the \textit{reverse triangle inequality}; see \eqref{eq:reverse:triangle}. Similarly, one can prove reverse isoperimetric-type inequalities in Lorentzian signature, which can be applied to questions about black holes and cosmology, as we explain in more detail later. 
\vspace{0.2cm}

The first result in this direction is due to Bahn and Ehrlich \cite{BE99}, who obtained a Lorentzian analogue of a special case of the inequality of Lions and Pacella \cite{LP90}: The area $A(S)$ of a compact connected smooth spacelike achronal (hence acausal) hypersurface $S$ with piecewise smooth boundary in the chronological future $I^+(O)$ of the origin $O$ in standard $(n+1)$-dimensional Minkowski space $\Lo^{n+1}$, $n\in\N$, is bounded \textit{from above} in terms of the volume $V(C(S))$ of the past cone $C(S)=\{tv \mid t\in (
0,1],\,v\in S\}$. More precisely, in terms of its isoperimetric deficit, their inequality can be stated as
\begin{equation}\label{eq:BE_ineq}
    \delta_{BE} (S)\coloneqq  (n+1) V(C(\pi(S)))^{\frac{1}{n+1}}\frac{V(C(S)) ^{\frac{n}{n+1}}}{A(S)} - 1 \geq 0 \,,
\end{equation}
where $\pi: S\to \HH^n$ denotes the radial projection to the (spacelike upper unit) hyperboloid $\HH^n$ centered at $O$. Intuitively, achronality resp.~acausality means that the hypersurface $S$ is visited at most once by every ``timelike'' resp.~``causal'' spacetime observer, who cannot resp.~can reach the speed of light. 
For precise definitions, in particular of the area and volume, we refer to Section \ref{sec:preliminaries} here and in the following.

The Bahn--Ehrlich inequality \eqref{eq:BE_ineq} is moreover sharp and rigid. Equality is achieved if and only if the hypersurface $S$ is contained in a (spacelike upper) hyperboloid $\HH^n_t$ of some radius $t>0$ centered at $O$. In particular, like the reverse triangle inequality as treated in \cite{LLS21}, the Bahn--Ehrlich inequality qualifies for a stability analysis in terms of its isoperimetric deficit~\eqref{eq:BE_ineq}. 
We introduce the Fraenkel asymmetry of $S$ as
\begin{equation}\label{eq:Frae}
  A_F(S)
      \coloneqq  \frac{V( C(S) \Delta B_{t}(\R_+S))}{V(C(S))}\, ,
\end{equation}
where $B_t(\R_+S)$ is the intersection of the cone $\R_+ S$ with the past of $\HH^n_t$, and where the value $t>0$ is chosen such that $V(B_t(\R_+S))=V(C(S))$. A discussion on the Fraenkel asymmetry and its properties can be found in Subsection \ref{subsec:dist_func}.

Given a domain $\Omega \subset \HH^n$ as defined in Subsection \ref{sub:conical_minkowski_cauchy_hypersurface}, we call the cone $M= \R_+ \Omega$ a \emph{conical Minkowski spacetime}. With its inherited Lorentzian structure, $M$ is a special Minkowski cone in the sense of \cite[Section~2]{AGKS23} and a Robertson--Walker warped product spacetime (with boundary); cf.~\cite[Chapter~12]{On83}. We can now state our first main result -- a stability estimate for the Bahn--Ehrlich inequality on an extended class of admissible hypersurfaces $S$. Here $\mu$ denotes the measure on $\HH^n$ induced by its natural hyperbolic metric as discussed in Subsection \ref{sub:basics}.

\begin{thm}[Stability for a generalized Bahn--Ehrlich isoperimetric inequality]\label{thm:BE_L1_stability} Let $M$ be a conical Minkowski spacetime in $\Lo^{n+1}$, $n\in \N$. Every achronal Lipschitz hypersurface $S$ of $M$ with $V(C(S))<\infty$ and $\mu(\pi(S))<\infty$ satisfies 
\[
			A_F(S)^2 
            \leq \frac{2(n+1)^2}n \delta_{BE}(S) \, .
\]
Moreover, the exponent $2$ of the Fraenkel asymmetry and the constant $2(n+1)^2/n$ are optimal.
\end{thm}

\addtocounter{thm}{-1}
\begin{remsthm}
(i) By a {\it (Lipschitz) hypersurface} we mean a nonempty codimension one $C^{0}$-($C^{0,1}$-)submani\-fold of $M$ with boundary. Every achronal set with empty edge, i.e.~empty boundary in a causal sense (see e.g.~\cite[Definition~14.23]{On83}), in particular every Cauchy hypersurface as defined in Subsection~\ref{sub:cauchy_hypersurface}, and every achronal hypersurface with Lipschitz boundary is
automatically a Lipschitz hypersurface; see, for instance, \cite[Theorem~2.147]{Mi19} and Lemma~\ref{lem:graph}, (iv), respectively. Topological hypersurfaces can be quite wild and of positive measure; see \cite{Os03}, for a classical example. By our ``Lipschitz'' assumption we effectively suppose that the boundary of $S$ and its image in $\HH^n\subset \Lo^{n+1}$ under the radial projection are null sets so that they can essentially be ignored.

(ii) Note that, in contrast to \cite{FI13}, the constant in our stability estimate depends only on the dimension and not on an ambient cone. The dimensional asymptotics of the constant coincide with the one of the quantitative Euclidean isoperimetric inequality in the nearly spherical case in \cite{CL12} and of the quantitative Sobolev inequality in \cite{DE+25}.

(iii) The notion of an achronal subset depends on the ambient spacetime. Every achronal subset of the conical Minkowski spacetime $I^+(O)=\R_+ \HH^n$, also known as \emph{Milne universe} \cite{Mi35}, that is contained in a conical Minkowski spacetime $M\subset I^+(O)$ is also an achronal subset of $M$, but the converse is not true -- unless $M$ is convex; see Subsection~\ref{sub:conical_minkowski_cauchy_hypersurface}. In particular,  the Bahn--Ehrlich isoperimetric inequality still holds for such achronal hypersurfaces under the weaker assumptions stated in the theorem, and in our proofs we can always assume that $M=\R_+S$. In Theorem \ref{thm:Bahn_Ehrlich_gen} we formulate the Bahn--Ehrlich isoperimetric inequality in even more generality. 

(iv) Since our Lipschitz hypersurface $S$ is nonempty, both $\mu(\pi(S))$ and $V(C(S))$ are positive.
Similarly to the non-vanishing and finiteness of the perimeter for the Euclidean isoperimetric inequality as in \cite[Theorem 1.1]{FMP08}, the conditions $0<V(C(S))<\infty$ and $ 0<\mu(\pi(S))<\infty$  guarantee well-definedness of the quantities involved.  However, in an asymptotic sense, the finiteness conditions can be dropped, and in such an asymptotic formulation the theorem also applies to Cauchy hypersurfaces of the Milne universe $M=I^+(O)$; see Subsection \ref{subsec:asymptotic}.
\end{remsthm}

 The original proof by Bahn and Ehrlich of their Lorentzian isoperimetric inequality is~based on a (Lorentzian) Brunn--Minkowski-type inequality, which is also the  approach in the related Euclidean case in \cite{LP90}. In Subsection \ref{subsec:func_pf} and \ref{subsec:geom_pf}, we provide simpler direct proofs of the isoperimetric inequality by Bahn and Ehrlich -- two functional-analytic ones and a geometric one, respectively.
 While all proofs might, in principle, serve as a starting point for a proof of Theorem \ref{thm:BE_L1_stability}, we quantify the geometric one, whose inductive step is based on an argument that appears in the equality discussion in \cite{BE99}. The idea to represent the hypersurface as a graph over a subset of the hyperboloid $\HH^n$ leads us to a situation similar to the Euclidean case considered by Fuglede \cite{Fu89} with only nearly spherical sets as competitors. However, in contrast to \cite{Fu89}, 
we do not impose a uniform $W^{1,\infty}$-bound, so we are \textit{not} restricted to ``nearly hyperbolic'' hypersurfaces in that sense. 
 
  In this functional-analytic setting,  our proof of Theorem \ref{thm:BE_L1_stability} is given in Subsection \ref{subsec:proof_BE_stability}. It relies on a decomposition into a sublevel and a superlevel set, which boils the proof down to a quantitative Minkowski-type inequality; see Subsection \ref{subsec:quantitative_Minkowski_ineq}. Here we exploit that, unlike in the Euclidean case, in our Lorentzian setting the surface area and the spacetime volume both behave additively under decompositions.  In Subsection \ref{subsec:sharpconst}, we give an example showing that in the theorem, like in the Euclidean case, the exponent $2$ of the Fraenkel asymmetry is optimal, even for hypersurfaces in a fixed smooth conical Minkowski spacetime. Moreover, we prove sharpness of the stability constant. Complementary stability results and their stability constants are discussed in Appendix \ref{app:compl_pf}.

Lambert and Scheuer \cite{LS21} were able to extend the Bahn--Ehrlich isoperimetric inequality to generalized Robertson--Walker spaces using a constrained mean curvature flow; see also \cite{Ba99,ACKW09} for previous results. Our approach extrapolates naturally to such settings, as will be discussed elsewhere.
\vspace{0.2cm}

A different isoperimetric-type inequality for acausal hypersurfaces has recently been obtained by Cavalletti and Mondino \cite{CM25} in the much more general setting of Lorentzian pre-length spaces satisfying timelike Ricci curvature lower bounds in a synthetic sense via optimal transport; see  \cite{KS18,CM22}.

As opposed to the Bahn--Ehrlich inequality \cite[Example 1.1]{CM25}, their approach has the advantage that it provides an isoperimetric-type inequality for the spacetime volume bounded by two nontrivial hypersurfaces. In the special setting of the Bahn--Ehrlich result above with only one nontrivial compact acausal hypersurface $S$, their inequality can be stated in terms of its isoperimetric deficit as
\begin{equation}\label{eq:CM_ineq}
		\delta_{CM} (S)\coloneqq  (n+1)\frac{ V(C(S))}{A(S)\cdot \inf |S|}-1\geq 0\, ,
\end{equation} 
where we abbreviate $\inf |S| \coloneqq \inf_{v \in S} |v|$ and where $|v|$ is the Lorentzian norm of $v$; see Subsection~\ref{sub:basics}. Also the Cavalletti--Mondino inequality is sharp and rigid: As before equality in \eqref{eq:CM_ineq} is attained if and only if the hypersurface $S$ is contained in $\HH^n_t$ for some $t>0$. 

In the following proposition, we relate the Bahn--Ehrlich (BE) inequality \eqref{eq:BE_ineq} and the special Cavalletti--Mondino (CM) inequality \eqref{eq:CM_ineq} in terms of the \textit{relative volume excess}
\begin{equation*}
   \mathcal E(S)\coloneqq \frac{1}{n+1}\frac{V\left(C(S)\right)-V(B_{\inf |S|}(\R_+S))}{V(C(S))}\,,
\end{equation*}
which satisfies $(n+1)\mathcal E(S)\in [0,1]$. Here the dimensional constant in $\mathcal E(S)$ is included to lighten the presentation.

 \begin{prp}[Relation between isoperimetric deficits]\label{prp:deficitbound}
 Let $M$ be a conical Minkowski spacetime in $\Lo^{n+1}$, $n\in\N$. Every achronal Lipschitz hypersurface $S$ of $M$ with $V(C(S))<\infty$ and $\mu(\pi(S))<\infty$ satisfies 
\begin{equation*}
\delta_{CM}(S) =\frac{1+\delta_{BE}(S)}{(1-(n+1)\mathcal E(S))^{\frac{1}{n+1}}}-1
\geq 
\mathcal E(S) +  (1+\mathcal  E(S))\delta_{BE}(S) \,.
\end{equation*}
\end{prp}

 This observation on the relation between isoperimetric deficits is a simple consequence of Bernoulli's inequality in the form $1-(n+1)\mathcal E(S)\leq (1+\mathcal E(S))^{-(n+1)}$ and is useful to prove stability of related Lorentzian isoperimetric inequalities. Indeed, we deduce the following stability result for the CM-inequality from Proposition \ref{prp:deficitbound} by combining a simple consequence of Lemma \ref{lem:asymmetry_comparsion}, namely
\begin{equation}\label{eq:RVEtoFr}
    A_F(S)\leq 2(n+1)\mathcal E(S)\,,
\end{equation} with the BE-inequality in the generality stated in Theorem \ref{thm:BE_L1_stability}. 

\begin{cor}[Stability for the CM-inequality]\label{cor:CM_L1_stability}   
Let $M$ be a conical Minkowski spacetime in $\Lo^{n+1}$, $n\in\N$. Every achronal Lipschitz hypersurface $S$ of $M$ with $V(C(S))<\infty$ and $\mu(\pi(S))<\infty$ satisfies \
\[
			A_F(S)  
            \leq  2(n+1) \delta_{CM} (S)\,.
\]
Moreover, the exponent $1$ of the Fraenkel asymmetry and the constant $2(n+1)$ are optimal.
\end{cor}

Sharpness of the stability exponent and the stability constant are discussed along with the corresponding properties for the BE-inequality in Subsection \ref{subsec:sharpconst}.

    While the BE-inequality \eqref{eq:BE_ineq} demonstrates the same stability behavior with a power $2$ of the Fraenkel asymmetry $A_F$ as the Euclidean isoperimetric inequality \cite{FMP08} and other geometric inequalities (see e.g.~\cite{FI13,FJ17,FMM18}),
    the CM-inequality \eqref{eq:CM_ineq} experiences a strengthened stability behavior. Strong stability results with a power of the distance equal to $1$ -- as found in Corollary~\ref{cor:CM_L1_stability} -- seem to be quite rare in the literature; see the two-dimensional result in \cite{FZ23}, though, and Subsection \ref{sub:functional_setting} for a comparison with the functional setting.
   
As another consequence of Proposition \ref{prp:deficitbound} and the BE-inequality, we obtain an improved isoperimetric-type inequality, which can be expressed in terms of a new isoperimetric deficit
\begin{equation}\label{eq:deltaCMstar}
     \delta_{CM}^*(S)\coloneqq \delta_{CM}(S)-\mathcal E(S)
     \geq 0\,.
 \end{equation} 
By Proposition \ref{prp:deficitbound} the isoperimetric deficits 
are related as
$$ \delta_{CM}(S)\geq \delta_{CM}^*(S)\geq \delta_{BE}(S) \geq 0\,.
$$ 
In particular, sharpness 
of the CM-inequality implies sharpness of the refined CM-inequality, and rigidity of the BE-inequality implies rigidity of the refined CM-inequality. 

Moreover, a consequence of Theorem \ref{thm:BE_L1_stability}  and Proposition \ref{prp:deficitbound} is that the refined CM-inequality, as opposed to the CM-inequality, exhibits the usual quadratic stability behavior.

\begin{cor}[Stability for the refined CM-inequality]\label{cor:CM_refined_L1_stability}  Let $M$ be a conical Minkowski spacetime in $\Lo^{n+1}$, $n\in\N$. Every achronal Lipschitz hypersurface $S$ of $M$  with $V(C(S))<\infty$ and $\mu(\pi(S))<\infty$ satisfies
\[
		A_F(S)^2 
            \leq  2\frac{(n+1)^2}{n}\delta^*_{CM}(S)\,.
\] 
Moreover, the exponent $2$ of the Fraenkel asymmetry and the dimensional dependence of the constant $2(n+1)^2/n$ are optimal.
\end{cor}

As we will see in Subsection \ref{subsec:sharpconst}, the optimal constant lies between $n+1$ and $2(n+1)^2/n$. Observe also that applying \eqref{eq:RVEtoFr} turns Corollary \ref{cor:CM_refined_L1_stability} into a higher-order stability estimate of Corollary \ref{cor:CM_L1_stability}.

As a consequence of their isoperimetric inequality, Cavalletti and Mondino discuss in \cite{CM25} upper bounds on the area
of acausal hypersurfaces inside the interior of a black hole and in cosmological spacetimes. 
In the latter case they work in a conical spacetime and interpret the tip of the cone (which corresponds to the origin $O$ in our special setting) as ``big bang'' and $\inf |S|$ roughly as the ``age'' of the universe. In the special case of a conical Minkowski spacetime, Corollary \ref{cor:CM_L1_stability} and \ref{cor:CM_refined_L1_stability} yield improved upper bounds.
\vspace{0.2cm}

Before, we mentioned that the Euclidean isoperimetric inequality in dimension $n\geq 3$ is not stable with respect to the Hausdorff distance.
A similar example, namely suitable hypersurfaces with a growing thin tentacle, shows that for $n\geq 3$ we cannot hope for a Hausdorff stability improvement of Theorem~\ref{thm:BE_L1_stability} with domain-independent constants. In the Euclidean case, Fuglede \cite{Fu89} showed that Hausdorff stability actually does hold in all dimensions when restricting to convex sets with Lipschitz boundary or to nearly spherical sets. It turns out that in the Lorentzian case we are naturally
in a ``nearly hyperbolic'' setting when working with achronal hypersurfaces (given as graphs) over a fixed bounded domain $\Omega \subset \HH^n$ with a Lipschitz boundary. Indeed,  a uniform $W^{1,\infty}$-bound on functions describing past volume-normalized achronal hypersurfaces over such a domain will be provided by the causal structure. This will allow us to upgrade our results measured in terms of the Fraenkel asymmetry to stability results measured in terms of a Hausdorff asymmetry -- at the cost of a constant that depends on the ambient spacetime and a worse exponent.

Thanks to the 
Lipschitz assumption on the domain $\Omega\subset \HH^n$, which we a priori define as an $n$-dimensional Lipschitz submanifold with boundary, we can in the following assume that $\Omega$ is closed without loss of generality as Lipschitz boundaries are null sets. Closedness of $\Omega$ implies that the conical  Minkowski spacetime $M=\R_+ \Omega$ satisfies a strong causality condition, namely \emph{global hyperbolicity}; see Subsection \ref{sub:cauchy_hypersurface}. The latter is equivalent to the existence of a \emph{Cauchy hypersurface} of $M$ as defined in Subsection \ref{sub:cauchy_hypersurface}, which intuitively is a subset visited exactly once by every spacetime observer. Since we consider bounded $\Omega$, we can equivalently assume that it is compact. Compactness of $\Omega$ guarantees, and is in fact equivalent to, \emph{spatial compactness} of $M$, i.e.~compactness of all Cauchy hypersurfaces of $M$, which under these assumptions 
are nothing else but achronal hypersurfaces $S$ of $M$ over $\Omega$; see Subsection~\ref{sub:graphs}.

Let $\mathcal{C}_M$ denote the space of Cauchy hypersurfaces in $M$. A metric on $\mathcal{C}_M$ can be defined by
$$
    d_\Delta(S_1,S_2) \coloneqq V(C(S_1) \Delta C(S_2))
$$
for $S_1,S_2 \in \mathcal{C}_M$. Following an approach by Bahn and Ehrlich \cite[Section~5]{BE99}, we introduce another natural, Hausdorff-type metric, denoted by $d_J$, on $\mathcal{C}_M$ as follows.

\begin{prp}\label{prp:hausdorff_metric} Let $M$ be a spatially compact conical Minkowski spacetime in $\Lo^{n+1}$, $n\in \N$, with induced Lorentzian distance $d$. Then a metric $d_J:\mathcal{C}_M \times \mathcal{C}_M \rightarrow [0,\infty)$ on $\mathcal{C}_M$ is defined by
\begin{equation}\label{eq:hausdorff-type-metric}
     d_J(S_1,S_2) \coloneqq \sup_{u\in S_1, v \in S_2} (d(u,v)+d(v,u)) \,.
\end{equation}
\end{prp}

 Note that $d(u,v) = |v-u|$ if $M$ is convex and $v$ is in the chronological future of $u$, but that the Lorentzian distance $d$ 
is more complicated in general. The proof of this proposition as well as an interpretation of $d_J$ as a Hausdorff-type metric is given in Subsection~\ref{subsec:Hausdorff}. For an extension of this proposition to more general spacetimes such as Lorentzian pre-length spaces and a discussion on properties of the resulting metric spaces, we refer to \cite{LP26a}.

With the metric $d_J$ at hand we define the Hausdorff asymmetry of a Cauchy hypersurface $S$ of $M$ as 
\begin{equation*} A_{H}(S) \coloneqq   \frac{\inf_{t>0} d_J(\HH^n_t \cap M, S)}{d_J(O,S)} \, ,
   \end{equation*}
where $d_J(O,S)\coloneqq \sup_{v\in S} 
|v|$, and show the following statement. 

\begin{thm} [Promotion  of Fraenkel asymmetry to Hausdorff asymmetry]
\label{thm:C^0_stability}
Let $M$ be a spatially compact conical Minkowski spacetime in $\Lo^{n+1}$, $n\in \N$. There is a constant $c_M>0$ such that for every Cauchy hypersurface $S$ of $M$ it holds that 
\begin{equation}\label{eq:HF}
     A_{H}(S)^{n+1} \leq  c_M A_F(S)  \, .
\end{equation}

If in addition $\pi(M)$ has a $C^1$-boundary, then there is a constant $\delta_{M}>0$ such that for every  Cauchy hypersurface $S$ of $M$ satisfying $A_F(S)\leq \delta_{M}$ the estimate \eqref{eq:HF} holds with $c_M=c_nV(B_1(M))$, where $c_n$ is an explicitly computable, dimensional constant.

Moreover, the exponent $n+1$ of the Hausdorff asymmetry and the dependence of $c_M$ on $V(B_1(M))$ are optimal.
\end{thm}
\addtocounter{thm}{-1}
\begin{remsthm}
(i) The optimality results at the end of the theorem are explicitly stated and proved in Subsection \ref{subsec:Linf_optimality} and \ref{subsec:domain_optimality}.

(ii) The implicit boundedness and Lipschitz regularity assumption on $\pi(M)$ are necessary. Without them, the constant $c_M$ can blow up; see Subsection \ref{subsec:domain_optimality}. 

(iii) The $C^1$-regularity assumption on $\pi(M)$ is necessary; see Subsection \ref{subsec:domain_optimality}. It is needed for the uniform interior cone condition to hold with a uniform domain-independent opening angle; see Lemma \ref{lem:uniform_in-radius}. Note that, as a consequence, in the $C^1$-case the constant $c_M=c_nV(B_1(M))$ only depends on the geometry of $M$ through $V(B_1(M))$. In contrast, the constant $\delta_M$ does not only depend on $V(B_1(M))$ but also on the geometry of the boundary of $M$; see \eqref{eq:choicedelM}.

(iv) In the spatially compact case the nonemptiness and finiteness assumptions from Theorem~\ref{thm:BE_L1_stability} are automatically satisfied for a Cauchy hypersurface $S$ of $M$.
\end{remsthm}

For the proof of Theorem \ref{thm:C^0_stability}, we will first show in Subsection \ref{subsec:Hausdorff} that
\begin{equation}\label{eq:hausdorff_identity}
d_J(\HH^n_t \cap M,S) = \sup_{v \in S} | t - |v| |
\,. 
\end{equation}
This will allow us to work in the functional setting as before, and it implies that $A_H(S) \leq 1/2$; cf.~\eqref{eq:AHreform}. By the latter, it suffices to prove a local version of Theorem~\ref{thm:C^0_stability}, i.e.~we can assume in addition that $A_F(S)$ is small. In the functional-analytic setting, the proof of Theorem~\ref{thm:C^0_stability} can then be reduced via an additional a priori estimate, Lemma \ref{lem:l1_to_sup}, to a statement about functions on a fixed compact domain $\Omega\subset \HH^n$, Proposition \ref{prp:L1_Linfty_bound_Lipschitz_smooth}. It says that for $L^\infty$-small locally $1$-Lipschitz functions we can find an upper bound for their $L^\infty$-norm in terms of an $L^1$-bound. The dependence on $\Omega$ in this bound can be made explicit, which leads to the statement about $C^1$-regular $M$ in Theorem~\ref{thm:C^0_stability}. Moreover, the constraints on the regularity of the boundary of $\Omega$, described in Remarks on Theorem \ref{thm:C^0_stability}, (ii) and (iii), originate from Proposition \ref{prp:L1_Linfty_bound_Lipschitz_smooth} and are discussed in more detail in Subsection \ref{subsec:domain_optimality}. Proposition \ref{prp:L1_Linfty_bound_Lipschitz_smooth} is proved via a localization argument in Subsection~\ref{subsec:prel_bdd}.

As a consequence of the isoperimetric bounds in Theorem \ref{thm:BE_L1_stability}, Corollary \ref{cor:CM_L1_stability}, and Corollary~\ref{cor:CM_refined_L1_stability}, we finally obtain a stability result with a Hausdorff asymmetry. In the following, we write $\delta$ if we mean that the statement holds for every deficit \begin{equation*}\delta\in\{\delta_{BE}^{1/2},\delta_{CM},(\delta_{CM}^*)^{1/2}\}\,.
\end{equation*} 

\begin{cor}[Hausdorff stability for Lorentzian isoperimetric inequalities]\label{cor:C0} 
   Let $M$ be a spatially compact conical Minkowski spacetime in $\Lo^{n+1}$, $n\in \N$. There is a constant $c_M>0$ such that for every Cauchy hypersurface $S$ of $M$ with $V(C(S))<\infty$ and $\mu(\pi(S))<\infty$ it holds that 
\begin{equation}\label{eq:HF_iso}
			A_{H}(S)^{n+1} \leq c_M \delta(S) \,. 
\end{equation}

If in addition $\pi(M)$ has a $C^1$-boundary, then there is a constant $\delta_{M}>0$ such that for every  Cauchy hypersurface $S$ of $M$ satisfying $A_F(S)\leq \delta_{M}$ the estimate \eqref{eq:HF_iso} holds with $c_M=c_nV(B_1(M))$, where $c_n$ is an explicitly computable, dimensional constant.
\end{cor}

Note that $c_M$ here, and thus $c_n$,  differs from the one in Theorem \ref{thm:C^0_stability} by the dimensional stability constant from the quantitative Lorentzian isoperimetric inequalities. While the ingredients of Corollary~\ref{cor:C0} are by themselves optimal with respect to the stability exponent as stated in the previous theorems, it is conceivable that the exponent of the Hausdorff asymmetry in Corollary \ref{cor:C0} can be improved; see Subsection~\ref{sub:discussion}. Our example in Subsection \ref{subsec:Linf_optimality} shows that the exponent of the Hausdorff asymmetry needs to be at least $n$ in case of the CM-inequality and at least $n/2$ in case of the BE-inequality and the refined CM-inequality.

Moreover, we point out that, by the corollary, if we have a family of Cauchy hypersurfaces of $M$ and their isoperimetric deficits converge to $0$, then after normalization these hypersurfaces converge $d_J$-uniformly, and hence also uniformly in the functional setting, to $\HH^n \cap M$.

\vspace{0.2cm}

Not only $c_M$ but also the constant $\delta_M$ in Theorem \ref{thm:C^0_stability} and Corollary \ref{cor:C0} can be made more explicit for $C^1$-regular $M$; see \eqref{eq:choicedelM}. We can use this explicit control to deduce asymptotic Hausdorff stability estimates in the case of infinite measure, for instance for the Milne universe $M=I^+(O)$. 

\begin{cor}[Asymptotic Hausdorff stability for non-compact domains]\label{cor:asymptotic_stability} 
There is an explicitly computable, dimensional constant $c_n>0$ such that for every Cauchy hypersurface $S
$ of $I^+(O) \subset \Lo^{n+1}$, $n\in \N$, we have 
$$A_H(S)^{n+1}
            \leq c_n  \liminf_{j\to\infty}  
            (\mu(\bar B_j(x_0)) \cdot \delta(S^{(j)}))$$ 
with Cauchy hypersurfaces $S^{(j)}= S\cap \R_+\bar B_j(x_0)$ of $\R_+ \bar B_j(x_0)$ for $j\in \N$ and some base point $x_0\in \HH^n$, where $\bar B_r(x)\subset \HH^n$ denotes the closed ball of radius $r$ and center $x$. 
\end{cor}

 In fact, while Proposition~\ref{prp:hausdorff_metric} is only formulated in the spatially compact case, both $d_J$ and the Hausdorff asymmetry $A_H$ extend naturally to the general case; see \cite{LP26a}. In particular, $A_H(S)=0$ still implies that $S$ is contained in a hyperboloid. While a similar asymptotic estimate in terms of the Fraenkel asymmetry is given in Corollary \ref{cor:asympt_Fr}, $A_F(S)=0$ does not imply that $S$ is contained in a hyperboloid. Note that up to dimensional constants $\mu(B_r(x_0))=A(B_r(x_0))$ asymptotically behaves like $e^{(n-1)r}$ if $n>1$ and like $r$ if $n=1$; cf.~Subsection \ref{sub:graphs}.

\subsection{Comparison with the stability of functional inequalities}
\label{sub:functional_setting} 

We complement the discussion on stability exponents for isoperimetric-type inequalities with results for functional inequalities, where this topic has attracted much attention in the past few years. Indeed, while in the seminal work \cite{BE91} a quadratic stability result in terms of the gradient $L^2$-norm as notion of distance to the set of optimizers has been established, (optimal) non-quadratic results were obtained only very recently; see, for instance, \cite{Fr22,BDS24,FP24a} for quartic stability results caused by degeneracy and \cite{FZ22, GLZ25,FP24b,FPR25} for non-quadratic results in the absence of a natural Hilbert space structure. In contrast to these results, the quantitative CM-inequality has a stability exponent with respect to the Fraenkel asymmetry that is \textit{smaller} than quadratic, which leaves room  for an improved stability result; see Corollary~\ref{cor:CM_refined_L1_stability}.

As mentioned before, Lorentzian inequalities have the opposite sign compared to their Euclidean counterparts. Reverse functional inequalities are prominent in the literature and arise mainly as lower dimensional extensions of well-known inequalities; see \cite{FKT22, CD+19, GW04, KP25}, for instance. In contrast, the sign-reversion of the Lorentzian inequalities is caused by the Lorentzian metric. Intriguingly, while explicit stability constants are usually rather difficult to obtain (see \cite{DE+25}), the (sharp) stability constant for the reverse Sobolev inequality can be found through a short and elegant argument; see \cite{GYZ25,Ko25}. Similarly, the stability constants for our Lorentzian inequalities are explicit.

\subsection{Structure of the paper} In Section \ref{sec:preliminaries} we recall some basic facts about Lorentzian geometry and introduce important notions such as achronal sets and conical Minkowski spacetimes.

In Section \ref{sec:func_repr} we characterize achronal Lipschitz hypersurfaces in terms of graphs over the unit hyperboloid and express relevant quantities in functional form. This allows us to give two direct, alternative proofs of the BE-inequality.

Our main stability result for the BE-inequality is proved in Section \ref{sec:BE_l1_stability_section}. More specifically, after motivating crucial ideas in yet another, geometric proof of the BE-inequality, we establish the quadratic stability result with respect to the Fraenkel asymmetry as given in Theorem~\ref{thm:BE_L1_stability}. Sharpness of our stability results with respect to the exponent of the Fraenkel asymmetry and the constant is discussed thereafter.

After introducing a Hausdorff-type metric on the space of Cauchy hypersurfaces, our results are promoted to stability results with respect to a Hausdorff asymmetry in Section~\ref{sec:C^0_stab}, where we prove Theorem~\ref{thm:C^0_stability}. Sharpness of our promotion theorem with respect to the exponent of the Hausdorff asymmetry and its dependence on the domain is discussed thereafter. In addition, we prove the asymptotic stability estimate Corollary \ref{cor:asymptotic_stability}.

Finally, the appendix supplements these results by providing complementary stability results of the BE-inequality in terms of quantitative H\"older inequalities.

\subsection{Sharpness of the results and open problems} \label{sub:discussion}
We close this introduction with a discussion on sharp exponents, sharp constants, and sharp dependencies of our results.

In the Euclidean case, the optimal constant for the quantitative isoperimetric inequality is explicitly known only in two dimensions and for asymptotically small Fraenkel asymmetries; see \cite[Subsection~4.2]{CL12}. Although we are able to show sharpness in Theorem \ref{thm:BE_L1_stability} and Corollary \ref{cor:CM_L1_stability} for the BE- and the CM-inequality, respectively, the question of determining the optimal constant for the refined CM-inequality in Corollary \ref{cor:CM_refined_L1_stability} is still open.

 The stability exponents in those results and Theorem \ref{thm:C^0_stability} are sharp. In contrast, the optimality of the exponents of $A_H$ in Corollary~\ref{cor:C0} seems to be an interesting problem we have not addressed yet. In Subsection \ref{subsec:Linf_optimality} we discuss how our examples that show optimality in Theorem \ref{thm:C^0_stability} only give proper lower bounds on the exponents. In this context, Fuglede's related stability results \cite{Fu89} in the nearly spherical or convex Euclidean setting can give first hints on what the optimal exponents might be.

As can be seen in Subsection \ref{subsec:dist_func} and \ref{subsec:Hausdorff}, the Fraenkel and Hausdorff asymmetry have a functional formulation in terms of the $L^{1}$-norm and the $L^\infty$-norm, respectively. (Note that different powers of functions in the respective distance arise naturally, but this can be unified; see Subsection \ref{subsec:Hausdorff}.)
Since we obtain an $L^1$-stability result with a stability constant that only depends on the dimension in Theorem \ref{thm:BE_L1_stability} and an $L^\infty$-stability result with a necessarily geometry-dependent stability constant in Corollary \ref{cor:C0}, it is natural to ask whether $L^p$-stability holds for any $1<p<\infty$ with geometry-\textit{independent} stability constant. We refer to Appendix~\ref{app:compl_pf} for some complementary $L^{p}$-type stability results with geometry-independent constants for a range of values of $p$. 
\newline

\textbf{Acknowledgements.} The first named author would like to thank Bernhard Leeb for support that allowed him to attend talks by Alessio Figalli and Andrea Mondino on stability of geometric inequalities and on Lorentzian isoperimetric inequalities, respectively. The second named author would like to thank Rupert Frank for introducing him to the stability of functional inequalities and for his continued support. We moreover would like to thank Andrea Mondino for pointing out the work of Bahn and Ehrlich. The second author acknowledges partial support through the Deutsche Forschungsgemeinschaft (DFG, German Research Foundation) grants FR 2664/3-1 and TRR 352-Project-ID 470903074 and through the Studienstiftung des
deutschen Volkes.

\section{Preliminaries}\label{sec:preliminaries}
\subsection{Some basic Lorentzian geometry} \label{sub:basics}
We recall some basics on Lorentzian geometry and refer, for instance, to \cite{On83} for more details. Let $\Lo^{n+1}$, $n\in \N$, be an $(n+1)$-dimensional Minkowski vector space with Minkowski inner product $\langle\cdot, \cdot \rangle$ of signature $-+\dots+$ and its standard topology. We denote the corresponding Lorentzian norm as $|v| \coloneqq \sqrt{|\langle v , v \rangle|}$. A vector $v\in \Lo^{n+1}$ is called \emph{spacelike} if $\langle v,v \rangle > 0$, \emph{causal} if it is not spacelike and \emph{timelike} if $\langle v,v \rangle < 0$. We impose a time-orientation on $\Lo^{n+1}$ by choosing a timelike vector $v_0 \in \Lo^{n+1}$ and calling a causal vector $v$ \textit{future-directed} if $\langle v,v_0 \rangle < 0$. The \emph{chronological future} of the origin $O$ in $\Lo^{n+1}$ is the set of future-directed timelike vectors in $\Lo^{n+1}$ and is denoted by $I^+(O)$. We note that every element in $I^+(O)$ induces the same time orientation on $\Lo^{n+1}$. Every pair of vectors $u,v \in I^+(O)$ satisfies the reverse triangle inequality
\begin{equation}\label{eq:reverse:triangle}
	|u| + |v| \leq |u+v| \, .
\end{equation}
For a subset $A\subset I^+(O)$ we abbreviate $$\inf |A|\coloneqq \inf_{v \in A} |v| \, .$$

A nondegenerate bilinear form on an oriented vector space induces a volume form on this vector space. In our setting, if $e_0,\ldots,e_{n}$ is an oriented orthonormal basis of $(\Lo^{n+1},\langle\cdot, \cdot \rangle)$, i.e.~$\langle e_j,e_k \rangle \in \{ \pm \delta_{jk}\}$, $j,k\in\{0,\dots, n\}$, then this volume form is given by
$e_0^* \wedge \ldots \wedge e_n^*,$ where $e_0^*,\ldots,e_{n}^*$ is a dual basis of $e_0,\ldots,e_{n}$. The volume form on $\Lo^{n+1}$ in turn induces a measure, which is independent of the chosen orientation. We denote this measure $V(D)$ for 
a measurable set $D\subset \Lo^{n+1}$ and call it the \textit{volume} of $D$. This volume coincides with the volume induced by the Euclidean inner product $\langle\cdot, \cdot \rangle_{\mathrm{Eucl}}$ defined by $\langle e_j,e_k \rangle_{\mathrm{Eucl}} = \delta_{jk}$; cf.~\cite{BE99}.

By a {\it hypersurface} of $\Lo^{n+1}$ we mean a nonempty codimension one $C^{0}$-submanifold of $\Lo^{n+1}$ with boundary. We call such a hypersurface Lipschitz if the submanifold with boundary is of class $C^{0,1}$. We refer the reader to \cite{EG15} for some background material on Lipschitz continuous objects. A Lipschitz hypersurface is called spacelike (non-timelike), if almost all of its tangent spaces contain only spacelike (non-timelike) nontrivial vectors.

 The volume form on $\Lo^{n+1}$ and the chosen time orientation induce an almost everywhere defined, locally uniformly bounded family of alternating $n$-forms on every non-timelike Lipschitz hypersurface $S$ of $\Lo^{n+1}$, which, in turn, induces a measure $\mu_S$ on $S$ via integration that is independent of the chosen time orientation. We denote the total measure $A(S)\coloneqq \mu_S(S)$ and call it the \textit{area} of $S$. Note that by construction our notions of volume and area are invariant under  isometries of the Minkowski inner product, in particular, under Lorentzian transformations. Using these notions, the definitions of the deficits \eqref{eq:BE_ineq}, \eqref{eq:CM_ineq}, and \eqref{eq:deltaCMstar} as well as of the Fraenkel asymmetry \eqref{eq:Frae} from the introduction extend naturally to the generality required for Theorem \ref{thm:BE_L1_stability}.

Of particular interest to us are the (upper) hyperboloids of radius $r>0$ given by 
\[
			\HH^n_r \coloneqq  \{v \in \Lo^{n+1} \mid \langle v,v\rangle = -r^2, \langle v,v_0\rangle < 0\} \subset I^+(O) \, ,
\] 
all of which are smooth spacelike hypersurfaces. The Minkowski inner product of $\Lo^{n+1}$ restricts to a smooth Riemannian metric on $\HH^n_r$ with constant negative sectional curvature $-1/r^2$. In particular, $\HH^n\coloneqq\HH^n_1$ is isometric to the $n$-dimensional hyperbolic space. We write $\pi: I^+(O) \rightarrow \HH^n$ for the radial projection, $d_{\HH}$ for the induced metric on $\HH^n$ (see also Subsection~\ref{subsec:induced_metrics}), and $\mu$ for the induced measure on $\HH^n$.

\subsection{Conical Minkowski spacetimes and achronal sets}\label{sub:conical_minkowski_cauchy_hypersurface}
By a {\it domain} in $\HH^n$ we mean a nonempty $n$-dimensional $C^{0,1}$-submanifold of $\HH^n$ with boundary. In particular, every domain has nonempty interior. While a domain is often assumed to be connected, all our results hold without this assumption. 

By a \textit{conical Minkowski spacetime} we mean a cone $M=\R_{+} \Omega$ for some domain $\Omega \subset \HH^n$. We call it $C^1$-regular if $\pi(M)$ has a $C^1$-boundary as a submanifold of $\HH^n$ with boundary. In this case $\pi(M)$ and $M$ are $C^1$-submanifolds with boundary. Sometimes it can be natural to consider $O$ as a singular point of $M$ (see below and \cite{LP26a}), but we do not include it in the definition here as it would cause unnecessary complications. We note that $M$ is convex if and only if $\Omega$ is convex. We further note that $\Lo^{n+1}$ induces the structure of a time-oriented Lorentzian $C^1$-manifold with boundary on every $C^1$-regular $M$. A conical Minkowski spacetime $M$ can be represented as a Lorentzian warped product of an interval and its defining domain $\pi(M)$; see \cite{AGKS23}, for instance. As described in the introduction, we define $B_t(M)\coloneqq C(\HH^n_t\cap M)\backslash \HH^n_t$, $t>0$. The Lorentzian geometry of conical Minkowski spacetimes will be discussed further in Subsection~\ref{subsec:time_sep}.

A locally Lipschitz continuous curve $\gamma:I\rightarrow \Lo^{n+1}$ on some interval $I$ is called \emph{causal} or \emph{timelike} if $\gamma'(t)$ satisfies the respective property for almost every $t\in I$. It is called \emph{future-directed} if $\langle \gamma'(t),v_0 \rangle < 0$ holds in addition almost everywhere. A causal curve (in a conical Minkowski spacetime $M$) is called \emph{inextendible} if none of its reparametrizations is the restriction of a causal curve (in $M$) that is defined on a larger interval. For instance, for $x\in \Omega$ the ray $(0,\infty)\ni t\mapsto t x \in M$ is an inextendible timelike curve in $M$. In particular, there is an inextendible timelike curve passing through every point of $M$.

A subset of a conical Minkowski spacetime $M$ is called \emph{achronal} resp.~\emph{acausal} (in $M$) if every timelike resp.~causal curve contained in $M$ intersects it \textit{at most} once. Note that acausal sets in $M$ and sets that are contained in $M$ but achronal in a larger ambient spacetime are also achronal in $M$. Therefore, in the proofs of Theorem~\ref{thm:BE_L1_stability} and its corollaries, we can always assume that $M=\R_+ S$. However, the converses do not hold in general. Every achronal Lipschitz hypersurface is non-timelike; see Lemma \ref{lem:graph}, also for a description of all achronal sets in a conical Minkowski spacetime.

\vspace{0.2cm}

The remaining subsections of the preliminaries only concern Proposition \ref{prp:hausdorff_metric}, Theorem \ref{thm:C^0_stability}, and its corollaries.

\subsection{Cauchy hypersurfaces}\label{sub:cauchy_hypersurface}
A subset of a conical Minkowski spacetime $M$ is called a \emph{Cauchy hypersurface (of $M$)} if it is intersected exactly once by every inextendible timelike curve in $M$. This notion is to some extent independent of the precise definition of a timelike curve; see Remark~\ref{rem:timelike_discussion}.
Every such Cauchy hypersurface is achronal (see Subsection \ref{sub:timelike-curves}) and a Lipschitz hypersurface (see~Lemma~\ref{lem:graph}), which justifies its naming.

Every conical Minkowski spacetime $M$ over a closed domain is globally hyperbolic (see \cite[Proposition~4.10]{AGKS23}) and has a Cauchy hypersurface. For instance, $M\cap \HH^n=\pi(M)$ is  a Cauchy hypersurface of $M$; see (the comment below) Lemma \ref{lem:cauchy_graph}.

A spacetime is called \emph{spatially compact} if it has a Cauchy hypersurface and all its Cauchy hypersurfaces are compact. As for globally hyperbolic Lorentzian manifolds, a conical Minkowski spacetime $M$ is spatially compact if and only if it has a compact Cauchy hypersurface, which happens if and only if its defining domain $\pi(M)$ is compact; see Lemma~\ref{lem:graph}, and also Lemma~\ref{lem:cauchy_graph} for a description of all Cauchy hypersurfaces of a spatially compact conical Minkowski spacetime. In the spatially compact case, central assumptions of singularity theorems by Hawking and Penrose are satisfied; see \cite{HP70} and \cite[Theorem 14.55B]{On83}, for instance. In line with this, a conical Minkowski spacetime is causally incomplete: every inextendible causal curve converges in the past towards the origin $O$, which can be thought of as a ``big bang''.

\subsection{Induced length metrics}\label{subsec:induced_metrics}

Let $(X,d)$ be an extended metric space, i.e.~a metric space in which $d$ is allowed to take the value $\infty$. The length of a (continuous) curve $\gamma : [a,b] \rightarrow X$ is defined as
$$
 L_d(\gamma) \coloneqq \sup \left\{ \sum_{i=0}^{N-1} d(\gamma(t_i),\gamma(t_{i+1})) \Bigg{|} \,N \in \N,\, a=t_0 < t_1 <\dots < t_N=b \right\} \, . $$ The curve $\gamma$ is called \emph{rectifiable} if $L_d(\gamma) < \infty$. The \emph{induced length metric} on $X$ is defined as
\[
        d_L(x,y) \coloneqq \inf  \left\{L_d(\gamma) \mid \gamma : [a,b] \rightarrow X \, \,C^0\text{-curve with }\gamma(a)=x, \gamma(b)=y \right\} \geq d(x,y) \, .
\]
The extended metric space $(X,d)$ is called \emph{intrinsic} if $d(x,y)=d_L(x,y)$ for all $x,y\in X$ with $d(x,y)<\infty$. It is called \emph{geodesic} if for all $x,y\in X$ with $d(x,y)<\infty$ there exists a curve $\gamma$ from $x$ to $y$ with $L_d(\gamma)=d(x,y)$. 
The pair $(X,d_L)$ is always an intrinsic extended metric space; see \cite[Proposition~2.3.12]{BBI01}, for instance.

Every rectifiable curve in a metric space can be parametrized by arclength; see \cite[Proposition~2.5.9]{BBI01}, for instance. With respect to this parametrization, it is a Lipschitz curve. In particular, in the definition of the induced length metric, we can equivalently work with the infimum over all Lipschitz curves from $x$ to $y$.

In an intrinsic extended metric space $(X,d)$, it is elementary to check that the following statements about a continuous function $f: X \rightarrow \R$ are equivalent.

(i)  $f$ is $L$-Lipschitz continuous (on a dense subset).

(ii) $f$ is locally $L$-Lipschitz continuous.

(iii) $|\nabla f(x_0)|\coloneqq\limsup_{x\rightarrow x_0} |f(x)-f(x_0)|/d(x,x_0) \leq L$ for all $x_0 \in X$.

We denote the induced length metric of a domain $\Omega \subset \HH^n$ by $d_\Omega$. We abbreviate $d_\HH\coloneqq d_{\HH^n}$,  and $L_\HH\coloneqq L_{d_\HH}$. Observe that the restrictions of $d_\HH$ and $d_\Omega$ to a convex subset of $\Omega$ coincide. In particular, we see that in the interior of such a domain $\Omega$, the conditions (ii) and (iii) above can be equivalently stated with respect to either $d_\HH$ or $d_\Omega$. We also note that a locally Lipschitz function $f:\Omega \rightarrow \R$ is almost everywhere differentiable by Rademacher's theorem and that the quantity in (iii) coincides with the norm of its gradient wherever it exists. If $\Omega \subset \HH^n$ is closed, then $(\Omega,d_\Omega)$ is proper and geodesic; see \cite[Corollary~2.5.20]{BBI01}, for instance.  

\subsection{Lorentzian geometry of conical Minkowski spacetimes}\label{subsec:time_sep}

In view of Proposition~\ref{prp:hausdorff_metric} and Theorem \ref{thm:C^0_stability}, we discuss the causal structure and the Lorentzian geometry of a conical Minkowski spacetime $M$ over a \emph{closed} domain $\pi(M)$ in this subsection. For two points $u,v \in M$ we write $u \leq v$ if there exists a future-directed causal curve from $u$ to $v$. The Lorentzian distance $d=d_M:M \times M \rightarrow [0,\infty]$  of $M$ between two points $u,v \in M$ is defined as the supremum of the Lorentzian lengths 
\[
 L(\gamma)=\int_I | \gamma'(t) | \, \mathrm dt \, 
\]
over all future-directed causal curves 
$\gamma$ from $u$ to $v$ that are contained in $M$ if $u \leq v$ and as $0$ otherwise. It satisfies the \emph{reverse triangle inequality}
\begin{equation*} d(u,v) + d(v,w) \leq d(u,w) \, \,\, \text{ for all } u,v,w \in M \text{ with } u\leq v\leq w \, .
\end{equation*}

If $u,v\in M$ satisfy $u \leq v$ and the straight segment from $u$ to $v$ is contained in $M$, then $d(u,v)=|u-v|$. In general, the Lorentzian distance of $u,v\in M$ with $u \leq v$ is given by
\[
   d(u,v)^2=  |v|^2 +|u|^2 - 2|u||v| \cosh(d_{\pi(M)}(\pi(u),\pi(v))) \, ,
\]
where $d_{\pi(M)}$ denotes the induced geodesic metric on $\pi(M)$ discussed in Subsection \ref{subsec:induced_metrics}; see \cite[Equation~(1)]{AGKS23}. In particular, the Lorentzian distance of $M$ is finite and continuous; see \cite[Proposition~2.2]{AGKS23}. With a Euclidean background metric, $(M,d,\leq)$ gives rise to a globally hyperbolic Lorentzian length space \cite[Theorem~4.8 and Proposition~4.10]{AGKS23}. Moreover, $(M,d)$ is geodesic, i.e.~for all $u,v\in M$ with $u \leq v$ there exists a future-directed causal curve $\gamma$ from $u$ to $v$ with $L(\gamma)=d(u,v)$; see \cite[Corollary~2.4]{AGKS23} or \cite[Theorem 3.30]{KS18}. In the convex case, the straight segment from $u$ to $v$ is such a curve. In fact, in Subsection~\ref{sub:timelike-curves} we will see that such a \emph{maximal} causal curve from $u$ to $v$ can be chosen to be timelike if $d(u,v)>0$ -- even in the non-convex case; see also Remark \ref{rem:timelike_discussion}.

Now the distance $d_J$ on the space of Cauchy hypersurfaces of $M$ can be defined as in Proposition \ref{prp:hausdorff_metric}. Further properties of $d_J$ are discussed in Section \ref{sec:C^0_stab} and in \cite{LP26a}.

\section{Functional representation for the Bahn--Ehrlich isoperimetric inequality} \label{sec:func_repr}

In this section we study a representation of achronal sets as graphs over the hyperboloid $\HH^n$. This provides the framework in which we prove Theorem \ref{thm:BE_L1_stability}; in particular, it allows us to obtain the BE-inequality for more general hypersurfaces. 

 \subsection{Timelike curves and geodesics} \label{sub:timelike-curves}
 To prove this graphical characterization  in Subsection \ref{sub:graphs}, we first need to quantify the regions that are reachable by timelike curves. 
 
\begin{lem}\label{lem:timelike_condition} 
 The following statements hold: \begin{enumerate}
    \item[(i)]Every future-directed causal curve $\gamma:[0,b] \rightarrow \Lo^{n+1}$, $b>0$, with positive length satisfies $$L_{\HH}(\pi \circ \gamma) < \ln \left( |\gamma(b)|/|\gamma(0)| \right)\,. $$
    \item[(ii)]For a curve $x: [0,b] \rightarrow \HH^n$ and $r_0,r_1 \in \R_+$, there is a future-directed timelike curve $\gamma:[0,b] \rightarrow \Lo^{n+1}$ with $(\pi \circ \gamma)([0,b]) = x([0,b])$, $\gamma(0)=r_0x(0)$, and $\gamma(b)=r_1x(b)
$ if
\begin{equation}\label{eq:ln_bound}
    L_{\HH}(x) < \ln ( r_1/r_0 )\,.\end{equation}
\end{enumerate}

\end{lem}

\begin{proof} Every Lipschitz curve $\gamma: [0,b] \rightarrow \Lo^{n+1}$ can be written as $\gamma(t)=r(t)x(t)$ for some Lipschitz curve $x:[0,b]\rightarrow \HH^n$ and some Lipschitz function $r:[0,b] \rightarrow \R_+$. Such a curve is future-directed if and only if $0 < -\langle \gamma'(t),x(t) \rangle = r'(t)$ for almost all $t$, 
and future-directed causal resp.~timelike if and only if \begin{equation}
    \label{eq:timelike_ineq}|x'(t)| \leq (\ln(r(t)))' \, \,\, \text{ resp. }\,\,\, |x'(t)| < (\ln(r(t)))'
\end{equation} for almost all $t$. If $\gamma$ is future-directed causal and $L(\gamma)>0$, then the strict inequality in \eqref{eq:timelike_ineq} holds on a set of positive measure. By integration, we see that  $r(b)> r(0)\exp(\int_0^b |x'(s)| \, \mathrm ds)$, which concludes (i).

Next, we prove (ii). If $x$ is constant, then $\gamma(t)=((b-t)r_0+tr_1)x(0)/b$ is a desired timelike curve. Otherwise, we reparametrize $x$ proportional to arclength.

We first assume that $x$ is a geodesic in $\HH^n$. If we have equality in \eqref{eq:timelike_ineq}, then $\gamma$ is lightlike, that is, $|\gamma'|=0$. By solving the ODE for $r$ associated to this equality with initial condition $r(0)=r_0$,
we find the lightlike straight segment $\tilde \gamma(t) \coloneqq r_0
\exp(\int_0^t |x'(s)| \, \mathrm ds) x(t)$ with $L_\HH(x) = \ln(|\tilde{\gamma}(b)|/r_0
)$. If \eqref{eq:ln_bound} holds, then we have $r_1> |\tilde \gamma(b)|$. Since $|x'(t)|>0$ almost everywhere, we can choose $r(t)$ such that $ \gamma(t)= r(t) x(t)$ parametrizes the (future-directed) timelike straight segment from $r_0 x(0)$ to $r_1 x(b)$. 

In the general case, we choose a geodesic $\tilde x:[0,b]\rightarrow \HH^n$ parametrized proportional to arclength with the same length as $x$ and construct a future-directed timelike straight segment $\tilde \gamma(t)=r(t)\tilde x(t)$ from $r_0 \tilde x(0)$ to $r_1 \tilde x(b)$ as before. Since \eqref{eq:timelike_ineq} depends on $x(t)$ only in terms of $|x'(t)|$, it is easy to verify that the product $ \gamma(t)=r(t)x(t)$ gives a curve that satisfies the desired properties. In this case we call $\tilde \gamma$ a \emph{development} of $\gamma$.
\end{proof}

A conical Minkowski spacetime $M$ is by definition \emph{causally path-connected}, i.e.~all points $u,v\in M$ with $u\leq v$ can be connected via a causal curve. In Subsection \ref{subsec:time_sep} we mentioned that such a curve can be chosen to be maximal if $\pi(M)$ is closed. The following corollary shows that the curve can be chosen to be maximal and timelike if we assume in addition that $d(u,v)>0$. 

 \begin{cor}[Length nondecreasing push-up] \label{lem:quant_push-up_curve}
 Let $M$ be a conical Minkowski spacetime in $\Lo^{n+1}$, $n\in\N$. For every future-directed causal curve $ \gamma:[0,b] \rightarrow M$ of positive length, there is a future-directed timelike curve $ \bar \gamma$ in $M$ with the same endpoints and $L(\bar\gamma) \geq L( \gamma)$. Moreover, given any neighborhood of $\gamma$, the curve $\bar \gamma$ can be chosen to lie inside this neighborhood.
\end{cor}
\begin{proof}
We write $ \gamma(t) = r(t)  x(t)$ as in the proof of Lemma \ref{lem:timelike_condition}. By Lemma \ref{lem:timelike_condition}, (i), condition \eqref{eq:ln_bound} is satisfied with $r_0=|\gamma(0)|$ and $r_1=|\gamma(b)|$. Let $\bar \gamma$ be the future-directed timelike curve constructed from $x$ as in the proof of Lemma \ref{lem:timelike_condition}, (ii). Since the development of $\gamma$ resp.~$\bar \gamma$ 
has the same length as $\gamma$ resp.~$\bar \gamma$ and since, by construction, the development of $\bar\gamma$ parametrizes a straight segment (and thus maximizes length), the first claim follows. 

By choosing a sufficiently fine subdivision of $[0,b]$ and applying Lemma \ref{lem:timelike_condition} to each subcurve separately, we can in addition guarantee that $\bar \gamma$ lies in a given neighborhood of $\gamma$.
\end{proof}

In fact, if $\pi(M)$ is closed, for $u,v \in M$ with $d(u,v)>0$ a maximal timelike geodesic from $u$ to $v$ can also be constructed directly from a minimal geodesic connecting $\pi(u)$ to $\pi(v)$ in $\pi(M)$ using Lemma \ref{lem:timelike_condition}. The corollary in particular says that a conical Minkowski spacetime $M$ is \emph{timelike path-connected}, i.e.~all points $u,v\in M$ that can be connected by a causal curve of positive length can also be connected by a timelike curve. Since radial rays in $M$ (see Subsection \ref{sub:conical_minkowski_cauchy_hypersurface}) define inextendible timelike curves, we see that every Cauchy hypersurface of $M$ is achronal; cf.~\cite[Lemma~5.11 (ii)]{LP26a}.

\begin{rem} \label{rem:timelike_discussion}
Note that our definition of a timelike curve, cf.~\cite[p.~146]{On83}, is stricter than the one used in \cite{AGKS23,KS18}, where a Lipschitz curve $\gamma$ is called timelike if $d(\gamma(t_1),\gamma(t_2)) > 0$ for all $t_1 < t_2$. For instance, the curve $t\mapsto(\sin(t),\cos(t), t)\in \Lo^3$ is timelike in the sense of \cite{AGKS23,KS18} but not in the present sense. For the weaker definition of timelike curves, the existence of timelike geodesics also follows from \cite[Corollary 4.9]{AGKS23} or \cite[Theorem~3.18]{KS18}. However, 
both definitions of a timelike curve lead to the same notion of a Cauchy hypersurface. Indeed, with the preceding corollary, one can show that a Cauchy hypersurface $S$ of a conical Minkowski spacetime $M$ is intersected exactly once by every inextendible causal curve $\gamma$ in $M$ satisfying $d(\gamma(t_1),\gamma(t_2)) >0$ for all $t_1< t_2$; see \cite[Remark 5.7]{LP26a} and cf.~\cite[Lemma~14.29 and Lemma~14.30]{On83}. \end{rem}

\subsection{Graphs and functional representations}
\label{sub:graphs} Away from its edge, an achronal set is known to be a Lipschitz hypersurface with empty boundary; see e.g.~\cite[Theorem 2.147]{Mi19}. In our setting, we can represent a general achronal set $S$ as a graph over the subset $\pi(S)$ of the hyperboloid $\HH^n$ as follows. Recall that $|\cdot|$ denotes the Minkowski norm in $\Lo^{n+1}$ and the hyperbolic norm in case of a tangent vector of $\HH^n$, respectively.

\begin{lem}\label{lem:graph} Let $\Omega \subset \HH^n$ be a domain. A subset $S$ of $M=\R_+\Omega$ is achronal in $M$ if and only if it has the form $$
    S=S_f \coloneqq \{ f(x)x\mid x \in \pi(S) \} $$ 
for some locally Lipschitz function $f: \pi(S) \rightarrow \R_+$ whose logarithm $\ln(f)$ is $1$-Lipschitz continuous with respect to the intrinsic metric of $\Omega$. 
In this case, we obtain the following:
\begin{enumerate}
    \item[(i)] The gradient bound \begin{equation*}
    |\nabla f|\leq |f|
\end{equation*} holds almost everywhere on $\pi(S) \subset \HH^n$. 
\item[(ii)] The preimage in $S$ of the interior of $\pi(S)\subset \HH^n$ under $\pi$ is a Lipschitz hypersurface.
\item[(iii)] The radial projection $S \rightarrow \pi(S)$ is the restriction of a locally bi-Lipschitz homeomorphism of $M$; in particular, the set $S$ is a (Lipschitz) hypersurface if and only if $\pi(S)$ is a (Lipschitz) hypersurface. 
\item[(iv)] If $S$ is a hypersurface whose boundary (as a manifold with boundary) is a $C^{0,1}$-submanifold, then it is even a Lipschitz hypersurface.
\end{enumerate}
\end{lem}

\begin{proof}Suppose first that $S$ is achronal. By definition, every $\R_+$-ray in $I^+(O)$ 
intersects the achronal set $S$ at most once, and so it is a graph $S=S_f$ over $\pi(S)$. If $\ln(f)$ were not $1$-Lipschitz continuous with respect to the intrinsic metric of $\Omega$, we could find $x_0,x_1 \in \pi(S)$ and a Lipschitz curve $x:[0,b] \rightarrow \Omega$ from $x_0$ to $x_1$ with $L_\HH (x) < \ln (f(x(b))/f(x(0)))$. Then Lemma \ref{lem:timelike_condition}, (ii), would provide a future-directed timelike curve in $M$ intersecting $S_f$ twice, a contradiction. Conversely, if $\ln(f)$ is $1$-Lipschitz with respect to the intrinsic metric of $\Omega$, then Lemma \ref{lem:timelike_condition}, (i), shows that no timelike curve in $M$ can intersect $S_f$ twice. The desired equivalence follows.

Claim (i) is a consequence of the $1$-Lipschitz continuity of $\ln(f)$ with respect to $d_\Omega$, the equivalences in Subsection \ref{subsec:induced_metrics}, and Rademacher's theorem. As a consequence of the characterization of achronal hypersurfaces, we deduce Claim (ii). By McShane's extension theorem \cite{Mc34}, $\ln(f)$ can be extended to a map $\ln(\bar f):\Omega \rightarrow \R$ that is $1$-Lipschitz continuous with respect to $d_\Omega$. This map is then also locally Lipschitz continuous with respect to $d_\HH$; see Subsection~\ref{subsec:induced_metrics}. We obtain a locally bi-Lipschitz homeomorphism $
\bar F:M \rightarrow M$, $ \bar F(v) = \bar f(v/|v|) \,v$, that maps $\pi(S)$ to $S$. Then Claim (iii) follows. If $S$ is a hypersurface, then the inverse of $\bar F$, which restricts to the radial projection on $S$, maps its interior (boundary) to the interior (boundary) of $\pi(S)$ by the invariance of domain theorem. Hence, if the boundary of $S$ is in addition a Lipschitz submanifold, then so is the boundary of $\pi(S)$. But then also the codimension one $C^0$-submanifold $\pi(S)$ is a Lipschitz submanifold with boundary. By Claim (iii), this concludes Claim~(iv). 
\end{proof}

\begin{lem}\label{lem:cauchy_graph}
 For a compact domain $\Omega \subset \HH^n$, the set of Cauchy hypersurfaces of the conical Minkowski spacetime $M= \R_+ \Omega$ is given by the set of graphs $S_f$ of
locally Lipschitz functions $f:\Omega \rightarrow \R_+$ for which $\ln(f)$ is $1$-Lipschitz continuous with respect to the intrinsic metric of~$\Omega$.
\end{lem}

In particular, the space of Cauchy hypersurfaces of a conical Minkowski spacetime $M$ with compact projection $\pi(M) \subset \HH^n$ is large, and every such Cauchy hypersurface is a compact Lipschitz hypersurface. For instance, every compact achronal smooth spacelike hypersurface $S \subset I^+(O)$, as used by Bahn and Ehrlich \cite{BE99}, is a Cauchy hypersurface of $M=\R_+ S$.

Moreover, note that while the statement of Lemma \ref{lem:cauchy_graph} is not true for merely closed domains $\Omega \subset \HH^n$, the argument in the proof below still shows that $\HH^n \cap M$ is a Cauchy hypersurface of $M$ in this case.

\begin{proof} Every Cauchy hypersurface of $M$ is achronal (see the discussion below Corollary \ref{lem:quant_push-up_curve}) and is thus of the form $S_f$ with $f$ as in Lemma \ref{lem:graph}. 

Conversely, suppose we are given a hypersurface $S_f$ with $f$ as in Lemma~\ref{lem:graph} defined on a compact domain $\Omega$ and an inextendible timelike curve $\gamma: I \rightarrow M$. We can assume that $0\in I$. If $|\gamma(0)|\leq f(\gamma(0)/|\gamma(0)|)$, then the inextendible timelike curve $\gamma$ has to leave the compact set $ C(S_f) \cap (\gamma(0) + \overline {I^+(O)})$  in forward time through $S_f$; cf.~\cite[Lemma 3.8]{LP26a}. Otherwise, $\gamma$ has to leave the compact set $ ([1,\infty) \cdot S_f) \cap (\gamma(0) - \overline {I^+(O)})$ in backward time through $S_f$. In any case, $\gamma$ intersects $S_f$. Since $S_f$ is moreover achronal by Lemma \ref{lem:graph}, the claim follows.
\end{proof}

Now we want to express the area $A(S)$ and the volume $V(C(S))$ for an achronal Lipschitz hypersurface $S=S_f$ in terms of the function $f$ provided by Lemma \ref{lem:graph}. To this end, we equip $\R_+ \times \HH^n$ with the product measure of the Lebesgue measure $\lambda^1$ on $\R_+$ and the hyperbolic measure $\mu$ on $\HH^n$. The diffeomorphism $\Phi: \R_+ \times \HH^n \rightarrow I^+(O)
$ defined by $\Phi(r,x)=rx$ satisfies $\mathrm d\Phi_{(r,x)} (1,y)=ry+ x$. Hence, $(\Phi^{-1})_* V = r^{n}\lambda^1\otimes \mu$. Then we find that
\begin{equation*}V(C(S))= \int_{\pi(S)} \int_0^{f(x)} r^{n}\,\mathrm d\lambda^1(r)\,\mathrm d\mu(x) = \frac{1}{n+1}\int_{\pi(S)} f(x)^{n+1} \,\mathrm d\mu(x)\,.
\end{equation*} 
In the remainder of this article, we suppress the Lebesgue measure in the notation and just write $\mathrm d t$, for instance.

Note that if $\pi(S)$ is a ball of radius $t^*>0$ in $\HH^n$ and the function $f$ is rotationally symmetric about the center $x_0$ of this ball, i.e.~$f(x)=r(d_{\HH}(x_0,x))$ for a function $r:[0,\infty)\to \R_+$, then the integral representation for $V(C(S))$ disintegrates to
\begin{equation*} V(C(S))=\frac{1}{n+1}\int_{0}^{t^*} r(t)^{n+1} \,\mathrm d\nu(t)\,,
\end{equation*}
where $\mathrm d\nu(t)=nV(B_1^{Eucl})\sinh^{n-1}(t)\,\mathrm d t$ and, as before, $B_1^{Eucl}$ denotes the $n$-dimensional Euclidean unit ball. 

Next, we turn to the area $A(S)$ of $S=S_f$. As the function $f$ is in general only locally Lipschitz continuous, all pointwise computations that are about to follow and involve the derivative of $f$ are to be understood in the almost everywhere sense. Set $F: \pi(S) \to \Lo^{n+1}$, $F(x)=f(x) x$. For a point $x\in \pi(S)$ and an orthonormal basis $v_1,\ldots,v_n$ of $T_x \pi(S)$ we have 
\[
    \mathrm dF_x v_j = \partial_{v_j}f (x) x + f(x) v_j\,, \qquad j\in \{1,\dots,n\}\,.
\]
The determinant of the Gram-matrix $(\langle  \mathrm dF_x v_j,  \mathrm dF_x v_k \rangle)_{j,k}$ is then given by
\[
    f(x)^{2n-2} (f(x)^2 - |\nabla f (x)|^2)= f(x)^{2n} (1- |\nabla \ln(f) (x)|^2)\,,
\]
which is nonnegative by Lemma \ref{lem:graph}, (i). Hence, we have \begin{equation*}
    A(S) = \int_{\pi(S)}  f(x)^n \sqrt{1- |\nabla \ln(f) (x)|^2} \,\mathrm{d} \mu(x)\,.
\end{equation*}
In particular, in the rotationally symmetric case, we have
\begin{equation*}
    A(S)=\int_{0}^{t^*} r(t)^{n}\sqrt{ 1- (\ln(r)'(t))^2}\,\mathrm d\nu(t)\,.
\end{equation*}
In the following, we also write $A(D)=\mu_S(D)$ for a general $\mu_S$-measurable subset $D$ of $S$ and note that the above formulas extend naturally to this case. This turns out to be useful in the proof of Theorem \ref{thm:BE_L1_stability}. More specifically, it enters a generalization of the Bahn--Ehrlich inequality, which we prove in the next subsection. 

\subsection{Two functional-analytic proofs of the Bahn--Ehrlich inequality}\label{subsec:func_pf}
As a first application of the discussed representation $S=S_f$, $f: \HH^n \supset \pi(S) \rightarrow \R$, of an achronal hypersurface $S$ provided by Lemma \ref{lem:graph}, we give two direct proofs for the BE-inequality without relying on a Lorentzian Brunn-Minkowski inequality or approximation arguments but based on Jensen's and Bernoulli's inequality, respectively. This approach moreover shows that the BE-inequality holds for general $\mu_S$-measurable subsets of an achronal hypersurface, which will be important for our stability proof in Subsection \ref{subsec:proof_BE_stability}.

\begin{thm}\label{thm:Bahn_Ehrlich_gen}
Let $M$ be a conical Minkowski spacetime in $\Lo^{n+1}$, $n\in \N$, $S$ an achronal Lipschitz hypersurface of $M$, and $D \subset S$ a  $\mu_S$-measurable subset. Then $\pi(D)$ is $\mu$-measurable, and it holds that 
\[
				A(D) \leq (n+1) V(C(\pi(D)))^{\frac 1 {n+1}} V(C(D))^{\frac{n}{n+1}} \, .
\] Moreover, if $V(C(D))<\infty$ and $V(C(\pi(D)))<\infty$, then equality holds if and only if $D\subset \HH^n_t$ almost everywhere for some $t>0$.
\end{thm}
Note that, in the situation of the theorem, $D$ is $\mu_S$-measurable if and only if $\pi(D)$ is $\mu$-measurable as the projection $\pi : S \rightarrow \pi(S)$ is a locally bi-Lipschitz homeomorphism.

In both proofs, the function $f$ is such that $S=S_f$ as in Lemma \ref{lem:graph}, and we set $g=f^{n+1}$.

\begin{proof}[First proof of Theorem \ref{thm:Bahn_Ehrlich_gen}] If $\mu(\pi(D))=0$, then the inequality holds trivially. If $\mu(\pi(D))\neq 0$,
 Jensen's inequality for the strictly concave function $(\,\cdot\,)^{n/(n+1)}$ with respect to the probability measure $\mu(\pi(D))^{-1}\mu$ on $\pi(D)$ immediately gives 
\begin{align*}
A(D)&\leq \int_{\pi(D)} g^{\frac{n}{n+1}}(x)\,\mathrm d\mu(x)\leq \mu(\pi(D))^{\frac 1{n+1}}\left(\int_{\pi(D)} g(x)\,\mathrm d\mu(x)\right)^{\frac{n}{n+1}} \\&= (n+1)V(C(\pi(D)))^{\frac 1 {n+1}}V(C(D))^{\frac n {n+1}}\,.\end{align*}
The equality claim follows by the equality discussion of Jensen's  inequality. Indeed, this implies that $g$ has to be constant almost everywhere, which also gives equality in the first inequality.\end{proof}

This proof was based on Jensen's inequality, which in this case coincides with H\"older's inequality for one constant function; see also Appendix \ref{app:compl_pf} for related stability results. Under the slightly stronger technical assumption  $V(C(D)),V(C(\pi(D)))\in \R_+$, we obtain another proof using Bernoulli's inequality.

\begin{proof}
[Second proof of Theorem \ref{thm:Bahn_Ehrlich_gen}]
     We set $\bar g=V(C(D))V(C(\pi(D)))^{-1}$, and $\phi= g-\bar g$. 
    In particular, $\int_{\pi(D)} \phi(x) \,\mathrm d\mu(x)=0$. Using Bernoulli's inequality $(1+a)^q\leq 1+q a$, $0<q<1$, $a>-1$, we find \begin{align*}A(D)&=\int_{\pi(D)}g(x)^{\frac{n}{n+1}}\sqrt{ 1- \frac{1}{(n+1)^2}|\nabla \ln(g)(x)|^2}\,\mathrm d\mu(x) \\&\leq \bar g^{\frac n{n+1}}\int_{\pi(D)} \left(1+\frac{1}{\bar g}\phi(x)\right)^{\frac{n}{n+1}}\,\mathrm d\mu(x)\leq \bar g^{\frac n{n+1}} \mu(\pi(D))\,.
\end{align*} By definition of $\bar g$, this is nothing else but the BE-inequality. The equality case follows from the equality discussion of Bernoulli's  inequality. \end{proof}

As a final preparation, let us discuss the Fraenkel asymmetry and a related notion in more detail, in particular in the functional setting.

\subsection{Notion of distance}
\label{subsec:dist_func}
Let $S$ be an achronal Lipschitz hypersurface in a conical Minkowski spacetime $M=\R_+S$ with $V(C(S))<\infty$ and $V(B_1(\R_+S))<\infty$. 

As in Lemma \ref{lem:graph}, let $f$ be such that $S=S_f$. The Fraenkel asymmetry can be written in terms of $f$ as
$$A_F(S)=\frac{V(C(S) \Delta B_t(\R_+S))}{V(C(S))}= \frac{\|f^{n+1}-t^{n+1}\|_{L^1(\pi(S))}}{\|f^{n+1}\|_{L^1(\pi(S))}}$$ with $t>0$ such that
$V(B_t(\R_+ S))=V(C(S))$, that is, $t^{n+1}=\mu(\pi(S))^{-1}\|f^{n+1}\|_{L^{1}(\pi(S))}$. 

To prove Corollary~\ref{cor:CM_L1_stability}, the slightly different notion of asymmetry given by
\begin{equation}\label{eq:AF_tilde}
    \tilde A_F(S)\coloneqq\inf_{t>0} \frac{V( C(S) \Delta B_t(\R_+S))}{V(C(S))}  =  \inf_{t>0}\frac{\|f^{n+1}-t^{n+1}\|_{L^1(\pi(S))}}{\|f^{n+1}\|_{L^1(\pi(S))}}
\end{equation}
becomes useful. Since this is a distance for the $(n+1)$-th power of $f$ in the $L^1$-norm, it can be regarded as an \textit{$L^1$-distance}. This asymmetry is equivalent to the Fraenkel asymmetry \eqref{eq:Frae} in the following sense. 

\begin{lem}\label{lem:asymmetry_comparsion}  Let $M$ be a conical Minkowski spacetime in $\Lo^{n+1}$, $n\in \N$, and $S$ an achronal Lipschitz hypersurface of $M$ with $V(C(S))<\infty$ and $\mu(\pi(S))<\infty$. Then the infimum in $\tilde A_F(S)$ is attained and satisfies 
\[
\tilde A_F(S) \leq  A_F(S) \leq 2 \tilde A_F(S)\,.
\]
\end{lem}

As mentioned in the introduction, this lemma together with the BE-inequality implies the quantitative CM-inequality; see Corollary \ref{cor:CM_L1_stability} and the discussion before that. Sharpness of the stability result is discussed in the next section along with the proof of the quantitative BE-inequality.

\begin{proof}
Exploiting the continuity of the measure, we observe that
    $$V(C(S)\Delta B_t(\R_+S))\to \infty \qquad \text{and} \qquad V(C(S)\Delta B_t(\R_+S))\to V(C(S))$$
as $t\to \infty$ and $t\to 0$, respectively. These limiting cases are bad competitors for the above infimum as we shall see next.

Set $\Omega=\pi(S)$. Since $\mu(\Omega)<\infty$, we can find a compact set $\Gamma\subset \Omega$ such that 
$\mu(\Omega\setminus \Gamma)<\mu(\Gamma)$. If we take $0<t\leq\min\{f(x)\mid x\in \Gamma\}$, where $f$ is such that $S=S_f$ as before, then $$V(B_t(\R_+S)\setminus C(S))\leq V(B_t(\R_+(\Omega\setminus\Gamma)))< V(B_t(\R_+\Gamma))\leq V(C(S)\cap B_t(\R_+S))$$ and thus
$$V(C(S)\Delta B_t(\R_+S))=V(C(S)\setminus B_t(\R_+S))+V(B_t(\R_+S)\setminus C(S))<V(C(S))\,.$$
By continuous dependence on $t$, this implies that the potential choices of $t$ that could reach the infimum are confined in a compact subset, and hence by compactness and continuity the infimum is attained.

The first inequality follows immediately as $A_F(S)$ gives a competitor for the infimum in $\tilde A_F(S)$. To prove the second inequality, we define $t_F$ by $V(B_{t_F}(\R_+S))=V(C(S))$ and $\tilde t_F$ as the $t>0$ at which $\tilde A_F(S)$ is attained. Note that $t_F$ satisfies \begin{equation}\label{eq:AF_prop}
    V(C(S)\Delta B_{t_F}(\R_+S))=2 V(C(S)\setminus B_{t_F}(\R_+S))=2 V( B_{t_F}(\R_+S)\setminus C(S))\,.
\end{equation}
If $\tilde t_F\geq t_F$, then $ V( B_{t_F}(\R_+S)\setminus C(S))\leq V(C(S) \Delta B_{\tilde t_F}(\R_+S))$. Similarly, if $\tilde t_F<t_F$,
then $V(C(S)\setminus B_{t_F}(\R_+S))\leq V(C(S) \Delta B_{\tilde t_F}(\R_+S))$. Combining both cases with \eqref{eq:AF_prop} concludes the second inequality in the statement of the lemma.
\end{proof}

\section{Stability of the Bahn--Ehrlich inequality}\label{sec:BE_l1_stability_section} 

Our goal in this section is to prove Theorem \ref{thm:BE_L1_stability}. In the first subsection we provide  another proof strategy for the BE-inequality and meanwhile discuss the key ingredient that we need in the proof of stability in Theorem \ref{thm:BE_L1_stability}.

\subsection{A geometric proof of the Bahn--Ehrlich inequality}\label{subsec:geom_pf}
 Bahn and Ehrlich \cite{BE99} proved their isoperimetric-type inequality \eqref{eq:BE_ineq} as a corollary of a Lorentzian Brunn--Minkowski inequality. Here we show that it can also be deduced from simple geometric considerations on a spacelike $n$-dimensional simplex $P\subset I^+(O)$.

Indeed, let $\Pi\subset \Lo^{n+1}$ be the supporting hyperplane of $P$ and let $h=d_J(O,\Pi\cap I^+(O))$ be the Lorentzian height of $C(P)$. Similarly to the Euclidean case, we have
\begin{equation}\label{eq:coneform}
        V(C(P)) = \frac{1}{n+1} hA(P)\,;   
\end{equation}
see \cite[Equation~(4.1)]{BE99}. As $\Pi$ is tangent to $\HH^n_h$, 
 we observe that $C(P)\subset \bar B_h(M)$  
for $M= \R_+ P$, 
and thus $V(C(P))\leq V(B_h(M))=h^{n+1}V(B_1(M))$. Together with \eqref{eq:coneform}, this implies that
\begin{align*}
    A(P) &= (n+1)h^{-1}V(C(P))^{\frac{1}{n+1}}  V(C(P)) ^{\frac{n}{n+1}} \leq (n+1)V(B_1(M))^{\frac{1}{n+1}}  V(C(P))^{\frac{n}{n+1}} \, , 
\end{align*}
    which is nothing else but the BE-inequality for $P$. 

Its extension to general hypersurfaces then follows from induction and approximation as employed by Bahn and Ehrlich \cite{BE99} for the Lorentzian Brunn--Minkowski inequality. More specifically, we decompose a polyhedral spacelike hypersurface $S\subset I^+(O)$ into hypersurfaces $S_1$ and $S_2$ with fewer simplices and suppose that the BE-inequality holds for $S_1$ and $S_2$. Then, with $\sigma_i = V(C(\pi(S_i)))$, $i=1,2$, and $V(C(\pi(S))) = \sigma_1 + \sigma_2$, we have
\begin{align}
			\notag A(S) &=A(S_1) + A(S_2) \\\notag
            &\leq (n+1) \left(\sigma_{1}^{\frac{1}{n+1}}  V(C(S_1))^{\frac{n}{n+1}} + \sigma_{2}^{\frac{1}{n+1}} V(C(S_2))^{\frac{n}{n+1}}\right) \\\notag &\leq (n+1)  (\sigma_{1}+\sigma_{2})^{\frac{1}{n+1}} (V(C(S_1))+V(C(S_2)))^{\frac{n}{n+1}} \\\label{eq:inductarg}
            & = (n+1) V(C(\pi(S)))^{\frac 1{n+1}} V(C(S))^{\frac{n}{n+1}}\,,
    \end{align}
 where we used an inequality of Minkowski \cite[Section~1.21]{BB61} in the penultimate step. This inequality can also be regarded as a Jensen inequality for the concave function $(\,\cdot\,)^{n/(n+1)}$ on a two-point probability space or as part of an inductive generalization of H\"older's inequality with $(n+1)$ entries; see \cite[Section~1.18]{BB61}. Hence, in the polyhedral case, the BE-inequality follows from the one for simplices above by induction. An extension to general achronal Lipschitz hypersurfaces can now be obtained by a standard approximation argument; cf.~\cite{BE99}.
 
Using $\inf |S|\leq h$,  the CM-inequality~\eqref{eq:CM_ineq} follows via a similar but simpler argument.

A quantitative version of the above Minkowski inequality is studied in the next subsection. It will be an essential piece in our stability proof.

\subsection{A quantitative Minkowski inequality} \label{subsec:quantitative_Minkowski_ineq}

The main step in our proof toward a stability result for the Bahn--Ehrlich inequality in terms of the Fraenkel asymmetry is the following quantitative Minkowski-type inequality.

\begin{prp}\label{prp:quantMink}
    For $\sigma_1,\sigma_2,V_1,V_2>0$ and $0<r<1$, we have
    $$(\sigma_1+\sigma_2)^{1-r}(V_1+V_2)^r-\sigma_1^{1-r}V_1^r-\sigma_2^{1-r}V_2^r\geq 2 r(1-r) \frac{(V_2\sigma_1-\sigma_2V_1)^2}{(\sigma_1+\sigma_2)^{2-(1-r)}(V_1+V_2)^{2-r}} \,.$$
\end{prp}
Note that the inequality holds for all $0<r<1$ and not only for $r=n/(n+1)$. Normalizing $\sigma_1+\sigma_2=1$ and $V_1+V_2=1$, the inequality is equivalent to
$$1-\sigma_1^{1-r}V_1^r-\sigma_2^{1-r}V_2^r\geq 2r(1-r)(V_2\sigma_1-\sigma_2V_1)^2\,.$$

Set $ \delta^*=V_2\sigma_1-\sigma_2V_1=V_2-\sigma_2=\sigma_1-V_1$ and $\delta=|\delta^*|$. Set $\Lambda=\sigma_1$ if $\delta^*>0$ and $\Lambda=\sigma_2$ otherwise. Then this normalized inequality is equivalent to the following lemma.

\begin{lem}
    For $0\leq \delta< \Lambda<1$ and $0<r<1$, we have
    $$1-\Lambda^{1-r}\left(\Lambda- \delta\right)^r-(1-\Lambda)^{1-r}\left((1-\Lambda)+ \delta\right)^r\geq 2r(1-r)\delta^2\,.$$
\end{lem}

\begin{proof}
  Note that the inequality holds trivially for $\delta=0$. Set 
    $$f_1(\delta)\coloneqq 1-\Lambda\left(1-\frac \delta{\Lambda}\right)^r-(1-\Lambda)\left(1+\frac \delta{(1-\Lambda)}\right)^r-2r(1-r)\delta^2\,.$$ Since $f_1(0)=0$, it suffices to prove for all $\delta \in [0,\Lambda)$ that
    $$f_1'(\delta)=r\left(1+\frac{\delta}{\Lambda-\delta}\right)^{1-r}-r\left(1-\frac{\delta}{1-(\Lambda-\delta)}\right)^{1-r}-4r(1-r)\delta\geq 0$$ to conclude the proposition. In turn, as $f_1'(0)=0$, we would like to differentiate another time and show that $f_1''(\delta)\geq 0$. However, this is not true for all $0<\delta<\Lambda$, mainly due to the nonlinear nature with respect to $\delta$ of the terms in the brackets. Therefore, we pass to a parametrization of the form $\delta(\Lambda-\delta)^{-1}=\theta \tilde \delta$ and $\delta (1-(\Lambda-\delta))^{-1}=(1-\theta) \tilde \delta$, which defines a bijection from $\{(\delta, \Lambda) \mid 0 \leq \delta<\Lambda<1\}$ to $\{(\tilde \delta, \theta) \mid 0\leq \tilde \delta<(1-\theta)^{-1}, 0<\theta<1\}$. 
    
     Define 
   $$f_2(\tilde \delta)\coloneqq \frac{1}{1-r}\left(\left(1+\theta \tilde \delta\right)^{1-r}-\left(1-(1-\theta)\tilde \delta\right)^{1-r}\right)-4\theta(1-\theta)\tilde \delta\,. $$
    As $r(1-r)f_2(\tilde\delta)= f_1'(\delta)$ and $f_2(0)=0$, it suffices to prove for all $\tilde \delta \in [0,(1-\theta)^{-1})$ that $$f_2'(\tilde \delta)=\theta\left(1+\theta \tilde \delta\right)^{-r}+(1-\theta)\left(1-(1-\theta)\tilde \delta\right)^{-r}-4\theta(1-\theta)\geq 0\,.$$
    By applying the binomial formula in the form $(a+b)^2\geq 4ab$,  $a,b\in \R$, twice, we see that \begin{align*}
        \Big(\theta\left(1+\theta \tilde \delta\right)^{-r}&+(1-\theta)\left(1-(1-\theta)\tilde \delta\right)^{-r}\Big)^2\geq 4 \theta(1-\theta)\left(\left(1+\theta \tilde \delta\right)\left(1-(1-\theta)\tilde \delta\right)\right)^{-r}\\&\geq (4\theta(1-\theta))^{1+r}\left((1-\theta)\left(1+\theta \tilde \delta\right)+\theta\left(1-(1-\theta)\tilde \delta\right)\right)^{-2r}\,.
    \end{align*} Take the square root and note that $4\theta(1-\theta)\leq 1$ and 
   $(1-\theta)(1+\theta \tilde \delta)+\theta(1-(1-\theta)\tilde \delta)=1$. This gives the desired inequality.
\end{proof}

Since  $f_1(0)=f_1'(0)=0$ and
$f_1''(0)=r(1-r)((\Lambda(1-\Lambda))^{-1}-4)> 0$ if and only if $\Lambda\neq 1/2$ by the binomial formula, Taylor's theorem implies that the constant $2r(1-r)$ is sharp.

We finally use Proposition \ref{prp:quantMink} to prove stability for the BE-inequality.

\subsection{Proof of Theorem \ref{thm:BE_L1_stability}
} \label{subsec:proof_BE_stability}
We work with the function $f :\pi(S) \rightarrow \R_+$ that represents our achronal hypersurface $S=S_f$, and we can assume that $M=\R_+ S$ as before. We denote $g=f^{n+1}$ and $\bar g=\mu(\pi(S))^{-1}\|g\|_{L^1(\pi(S))}$.

We split the domain $\pi(S)$ into
$\{g\leq \bar g\}$ and its complement. Define $D_1\coloneqq S\cap \R_+\{g\leq \bar g\}$ and $D_2\coloneqq S\cap \R_+(\pi(S)\setminus \{g\leq \bar g\})$ and set $\sigma_i=V(C(\pi(D_i)))$ and  $V_i=V(C(D_i))$ for $i\in\{1,2\}$. Note that $\sigma_1+\sigma_2,V_1+V_2\in (0,\infty)$ by assumption. Moreover, $\sigma_1, V_1>0$ by construction, and if $\sigma_2=0$ or $V_2=0$, then $g=\bar g$ almost everywhere and both sides of our quantitative estimate are $0$. Hence, we may assume $\sigma_1,\sigma_2,V_1,V_2>0$.

We obtain
\begin{align*}
    \frac{1}{n+1}\int_{\pi(S)} |\bar g-g| \,\mathrm d\mu&=  \frac{1}{n+1}\int_{\pi(D_1)} (\bar g-g) \,\mathrm d\mu+ \frac{1}{n+1}\int_{\pi(D_2)} (g-\bar g) \,\mathrm d\mu\\&=\frac{\sigma_1}{\sigma_1+\sigma_2} (V_1+V_2)-V_1+V_2-\frac{\sigma_2}{\sigma_1+\sigma_2} (V_1+V_2)=2\frac{V_2\sigma_1-\sigma_2V_1}{(\sigma_1+\sigma_2)}\,.
\end{align*}Note that 
$t^{n+1}=\bar g$ for $t>0$ satisfying $V(B_t(\R_+S))=V(C(S))$. Hence, the Fraenkel asymmetry can be written as
$$A_F(S)= \frac{V(C(S)\Delta B_t(\R_+S))}{V(C(S))}=\frac{1}{\|g\|_{L^{1}(\pi(S))}}\int_{\pi(S)}|g-\bar g|\,\mathrm d\mu=
2\frac{V_2\sigma_1-\sigma_2V_1}{(\sigma_1+\sigma_2)(V_1+V_2)}\,.$$
As $D_1$ and $D_2$ satisfy the assumptions of the BE-inequality in the version of Theorem~\ref{thm:Bahn_Ehrlich_gen}, the inductive argument given in \eqref{eq:inductarg} remains to hold. 
Applying Proposition \ref{prp:quantMink} to \eqref{eq:inductarg} with $r=n/(n+1)$ gives $$\frac{(n+1) V(C(\pi(S)))^{\frac 1{n+1}}V(C(S))^{\frac n{n+1}}}{A(S)}-1\geq  \frac{n}{2(n+1)} \frac{V(C(\pi(S)))^{\frac{1}{n+1}}V(C(S))^{\frac{n}{n+1}}}{A(S)}A_F(S)^2\,.$$
Another application of the BE-inequality on the right side leads to the desired stability result.

 The proof of the optimality of the exponent and of the constant is part of the next subsection. 
\hfill \qed

\subsection{Sharpness of the stability constants and exponents}\label{subsec:sharpconst}

In this subsection we show that the constants $2(n+1)^2/n$ in Theorem \ref{thm:BE_L1_stability} and $2(n+1)$ in Corollary \ref{cor:CM_L1_stability} are sharp, respectively, that is, there is   
a family $S_{\varepsilon}$ of achronal Lipschitz (in fact locally constant) hypersurfaces in a fixed smooth conical Minkowski spacetime such that $\lim_{\varepsilon\to 0}A_F(S_{\varepsilon})= 0$ 
and
\begin{equation*}
\limsup_{\varepsilon\to 0}\frac{\delta_{BE}(S_{\varepsilon})}{A_F(S_{\varepsilon})^2}=\frac{n}{2(n+1)^2}\qquad \text{resp.}\qquad  \limsup_{\varepsilon\to 0}\frac{\delta_{CM}(S_{\varepsilon})}{A_F(S_{\varepsilon})}=\frac{1}{2(n+1)}\,.
\end{equation*} 
Note that for the $S_\varepsilon$ that we are going to consider we can easily find such a conical Minkowski spacetime that contains all $S_\varepsilon$. Further note that this sharpness also implies that the exponent $2$ in Theorem \ref{thm:BE_L1_stability} and the exponent $1$ in Corollary \ref{cor:CM_L1_stability} are sharp, respectively. Our computations further yield sharpness of the exponent $2$ and optimal dimensional dependence of the stability constant for Corollary~\ref{cor:CM_refined_L1_stability}. However, contrary to the other Lorentzian isoperimetric inequalities above, we cannot prove that the constant is optimal by our methods.

On compact domains with two distinct path-connected components $A_1$ and $A_2$, we consider surfaces $S_\varepsilon$ that are constant on each component and are therefore Cauchy hypersurfaces. More specifically, we consider $$f_\varepsilon(x)\coloneqq \begin{cases}
   a_1\quad \text{for}\ x\in A_1\,, \\
    a_2\quad \text{for}\ x\in A_2\,,
\end{cases}$$ for some $a_1,a_2>0$, so $\nabla f_\varepsilon\equiv0$. Here $A_1$, $A_2$, $a_1$, and $a_2$ may depend on $\varepsilon$. Taking a sequence of disconnected Cauchy hypersurfaces is only for convenience. Indeed, by looking at domains with an arbitrarily thin and sufficiently long connecting neck such that $|\nabla \log(f_\varepsilon)|\leq 1$ along this neck, we could obtain a sequence of connected Cauchy hypersurfaces as well.

First, for the BE-inequality, take $\mu(A_1)=\mu(A_2)=1/2$ and $a_1=(1+\varepsilon)^{1/(n+1)}$ and $a_2=(1-\varepsilon)^{1/(n+1)}$ for some $\varepsilon>0$. A short computation yields
$$V(C(S_\varepsilon))=\frac 1{n+1}\quad\text{and} \quad A(S_\varepsilon)=\frac{1}{2}\left((1+\varepsilon)^{\frac n{n+1}}+(1-\varepsilon)^{\frac n{n+1}}\right)=1-\frac{n}{2(n+1)^2}\varepsilon^2+o(\varepsilon^2)\,.$$ Moreover, for $t>0$ given by $V(B_t(\R_+S_\varepsilon))=V(C(S_\varepsilon))$, we have $$V(C(S_\varepsilon)\Delta B_t(\R_+S_\varepsilon))= \varepsilon V(C(S_\varepsilon))\,,$$
so $A_F(S_\varepsilon)=\varepsilon$. Thus, we find $$\frac{\delta_{BE}(S_\varepsilon)}{A_F(S_\varepsilon)^2}=\frac{n}{2(n+1)^2}+o(\varepsilon)\,.$$

Note that, as $\inf|S_\varepsilon|=(1-\varepsilon)^{1/(n+1)}$ and $\mathcal E(S_\varepsilon)=\varepsilon/(n+1)$, we further find 
$$\frac{\delta_{CM}^*(S_\varepsilon)}{A_F(S_\varepsilon)^2}=\frac{1}{n+1}+o(\varepsilon)\,,$$ which is (asymptotically) off by a factor of $2$ when comparing it with the constant in Corollary~\ref{cor:CM_refined_L1_stability}, but still proves sharpness of the exponent $2$ and gives the optimal asymptotic behavior in $n$.

Next, for the CM-inequality, take $a_1=(1+\varepsilon)^{1/(n+1)}$, $a_2=1$, and $\varepsilon=\mu(A_1)=1-\mu(A_2)>0$. Then a short computation gives $$V(C(S_\varepsilon))=\frac 1{n+1}\left(1+\varepsilon^2\right)\,,\qquad A(S_\varepsilon)=(1+\varepsilon)^{\frac n{n+1}}\varepsilon+1-\varepsilon=1+\frac{n}{n+1}\varepsilon^2+o(\varepsilon^2)\,,$$
and $$V(C(S_\varepsilon)\Delta B_t(\R_+S_\varepsilon))= \frac{2}{n+1}\varepsilon^2+o(\varepsilon^2)\,.$$ As $\inf|S_\varepsilon|=1$, we obtain 
 $$\frac{\delta_{CM}(S_\varepsilon)}{A_F(S_\varepsilon)}=\frac{1}{2(n+1)}+o(\varepsilon)\,.$$

\section{Promotion to Hausdorff stability}\label{sec:C^0_stab}
 In this section we first prove Theorem \ref{thm:C^0_stability} and thereby upgrade the Fraenkel asymmetry in Theorem \ref{thm:BE_L1_stability}, Corollary \ref{cor:CM_L1_stability}, and  Corollary \ref{cor:CM_refined_L1_stability} to a Hausdorff asymmetry as stated in Corollary~\ref{cor:C0}. Afterwards, we establish the sharpness and optimality statements from the introduction and prove the asymptotic stability statement in Corollary \ref{cor:asymptotic_stability}.

\subsection{Hausdorff-type metric on the space of Cauchy hypersurfaces}\label{subsec:Hausdorff}

In a globally hyperbolic Lorentzian manifold $M$, the distance \eqref{eq:hausdorff-type-metric} introduced by Bahn and Ehrlich restricts to an extended Hausdorff-type metric on the space of Cauchy hypersurfaces; 
see \cite[Theorem 1.1]{LP26a} for more details and generalizations to synthetic Lorentzian settings.  
For the convenience of the reader, we discuss this construction here in the special case of a spatially compact conical Minkowski spacetime, as considered in this section. In this case the metric $d_J$ only takes finite values by continuity of the Lorentzian distance. To start with, we have the following variant of \cite[Lemma~5.1]{BE99}. By abuse of notation, we denote $$d_J(u,v)\coloneqq d(u,v)+d(v,u)\,,\qquad u,v\in M\,.$$

\begin{lem}[Definiteness]
Let $M$ be a conical Minkowski spacetime. Then Cauchy hypersurfaces $A,B$ of $M$ satisfy $d_J(A,B)=0$ if and only if $A=B$.
\end{lem}

\begin{proof} Suppose $A$ and $B$ are distinct. Then there is without loss of generality some $x \in A\backslash B$. Since $B$ is a Cauchy hypersurface, the timelike ray from the origin through $x$ intersects $B$ in a point $y\in B$ such that $d_J(A,B) \geq d_J(x,y) >0$. Conversely, since no timelike curve in $M$ intersects $A$ twice, we have $d_J(x,y)=0$ for all $x,y\in A$, so $d_J(A,A)=0$.
\end{proof}

Unlike Bahn and Ehrlich, in our setting we can show that $d_J$ satisfies the triangle inequality.

\begin{lem}[Triangle inequality] Let $M$ be a spatially compact conical Minkowski spacetime. Then Cauchy hypersurfaces $A,B,C$ of $M$ satisfy
\[
   d_J(A,C) \leq d_J(A,B) + d_J(B,C)\,.
\]
 \end{lem}
 
 \begin{proof} 
By compactness of $A$ and $C$ and continuity of the Lorentzian distance, there are $x \in A$ and $z \in C$ with $d_J(A,C)=d_J(x,z)<\infty$. We can assume that $d(x,z)>0$ without loss of generality. We consider the concatenation $\gamma$ of the timelike straight segment from the origin to $x$, a timelike geodesic from  $x$ to $z$ (see Corollary \ref{lem:quant_push-up_curve}), and a timelike ray starting at $z$ and going off to infinity. The curve $\gamma$ is then an inextendible timelike curve and hence intersects the Cauchy hypersurface $B$ in a unique point $y$. If $y$ lies on the timelike geodesic between $x$ and $z$, then
 \begin{equation*} d_J(A,C)= d_J(x,z) =d_J(x,y)+d_J(y,z) 
    \leq d_J(A,B) + d_J(B,C)\,.
\end{equation*} 
If $y$ lies on the segment from the origin to $x$, then $d(y,x)+d(x,z)\leq d(y,z)$ by the reverse triangle inequality,  and hence
\begin{equation*} d_J(A,C)=d_J(x,z)  \leq d_J(x,y) + d_J(y,z)
    \leq d_J(A,B) + d_J(B,C)\,.
\end{equation*} 
The case in which $y$ lies on the timelike ray starting at $z$ is analogous.
\end{proof}

To summarize, we see that for a spatially compact conical Minkowski spacetime $M$ the space of Cauchy hypersurfaces $(\mathcal C_M, d_J)$ is a metric space.

We finish this subsection with a proof of the special formulation \eqref{eq:hausdorff_identity} of $d_J$ when comparing Cauchy hypersurfaces with hyperbolic Cauchy hypersurfaces. The proof is based on the following statement.

\begin{lem}For $0\leq s<  t$ and $v\in M$, it holds that
\[
    d_J(v, \HH^n_t\cap M)\leq s\qquad \text{if and only if} \qquad \left| t- |v| \right| \leq s\,. 
\]
\end{lem}
\begin{proof} Since the Lorentzian distance on rays from the origin is simply given by the Lorentzian norm, we only need to show that  for $v\in M$ the supremum $d_J(v, \HH^n_t\cap M)$ is attained at $tv/|v|$. This is clear if $|v|=t$, since $\HH^n_t\cap M$ is achronal. For $|v|>t$ and each $w\in \HH^n_t\cap M$ with $d(w,v)>0$, the reverse triangle inequality \eqref{eq:reverse:triangle} yields
\[
   d\left( t\frac{v}{|v|},v \right) =   d(O,v) - t \geq d(O,w) +d(w,v) -t =d(w,v) \, .
\]
The case $|v|<t$ works analogously.
\end{proof}

Hence, for a Cauchy hypersurface $S$ of $M$ we have $d_J(\HH^n_t\cap M,S ) \leq s$ if and only if $\sup_{v \in S} | t - |v| | \leq s$.

It follows that $d_J(\HH^n_t\cap M,S) = \sup_{v \in S} | t - |v| |$ as claimed in \eqref{eq:hausdorff_identity}. Therefore, in the functional representation $S=S_f$, the Hausdorff asymmetry can be written as
\begin{equation}
    \label{eq:AH_funct}A_{H}(S)= \inf_{t>0}  \frac{d_J( \HH^n_t \cap M, S)}{d_J(O,S)} =  \frac{1}{d_J(O,S)}\inf_{t>0}\sup_{x\in \pi(S)} |f(x)-t|\,.
\end{equation}
       As this distance is taken with respect to the $L^\infty$-norm, it can be regarded as an \textit{$L^\infty$-distance}. Note that the exponents of $f$ in the Fraenkel and Hausdorff asymmetry differ. Indeed, we consider the $L^\infty$-norm of $f(x)-t$ in the Hausdorff asymmetry, whereas for the Fraenkel asymmetry we consider the $L^1$-norm of $f^{n+1}(x)-t^{n+1}$. However, they are still comparable as can be seen from the proof of Theorem \ref{thm:C^0_stability}, which also provides a stability result with the $L^\infty$-distance for the $(n+1)$-th power of $f$.

\subsection{Preliminary bounds on the \texorpdfstring{$L^\infty$}{Linfty}-norm in terms of the \texorpdfstring{$L^1$}{L1}-norm}\label{subsec:prel_bdd}

In order to prove Theorem \ref{thm:C^0_stability}, we provide bounds on the $L^\infty$-norm in terms of the $L^1$-norm assuming a Lipschitz condition in this subsection. In particular, we show the following statement. Recall that all our domains are assumed to be Lipschitz.

\begin{prp} \label{prp:L1_Linfty_bound_Lipschitz_smooth}
Let $\Omega\subset \HH^n$ be a compact (Lipschitz) domain. There are constants
$h_\Omega,c_\Omega>0$ such that for every continuous function $h:\Omega \to \R$ that is locally $1$-Lipschitz continuous in the interior of $\Omega$ with $\|h\|_{L^\infty(\Omega)}\leq h_\Omega$, it holds that
$$\|h\|_{L^\infty(\Omega)}^{n+1}  \leq  c_\Omega\|h\|_{L^1(\Omega)}\,.$$ 
Moreover, for $\Omega$ with $C^1$-boundary, the constant $c_\Omega$ can be replaced by a constant $c_n$ that only depends on the dimension.  
\end{prp}

\addtocounter{thm}{-1}
\begin{remsprp}(i)  The first part of the proposition holds verbatim for compact domains $\Omega$ (defined as $C^{0,1}$-submanifolds as before) in general Riemannian manifolds $N$.
Using compactness, this general case can be reduced to domains in $\R^n$ via a cover of $\Omega$ by a finite number of compact domains such that each member of the cover is a bi-Lipschitz homeomorphic image of a compact domain in $\R^n$ under the exponential map at some point. (In case $N=\HH^n$ no cover is required.) Indeed, for every continuous function $h:\Omega \rightarrow \R$ and every measurable set $\Omega'\subset \Omega\subset N$  with $x_0\in  \Omega'$ satisfying $|h(x_0)|=\|h\|_{L^\infty( \Omega)}$ 
it holds that \begin{equation}
\label{eq:localize_Linf_L1}\|h\|_{L^{\infty}( \Omega')}= \|h\|_{L^{\infty}(\Omega)}\quad \text{and}\quad\|h\|_{L^{1}( \Omega')}\leq \|h\|_{L^{1}(\Omega)}\,.
\end{equation}
Moreover, if the exponential map $\exp$ restricts to an $L$-bi-Lipschitz homeomorphism, $L=L_\Omega\geq 1$, from a compact domain $\tilde \Omega \subset \R^n$ 
 to a compact domain $\Omega' = \exp(\tilde \Omega) \subset \Omega$, then the function $\tilde h \coloneqq L^{-1} (h\circ \exp)$ maps $\tilde \Omega$ to $\R$ and satisfies
\begin{equation*}
    \|\tilde h\|_{L^{\infty}(\tilde \Omega)}= \frac 1 L \|h\|_{L^{\infty}(\Omega')}\quad \text{and}\quad
    \frac 1 {L^{n-1}} \|\tilde h\|_{L^{1}(\tilde \Omega)}\leq \|h\|_{L^{1}(\Omega')}\leq L^{n+1}\|\tilde h\|_{L^{1}(\tilde \Omega)}\,.
\end{equation*} Finally, by construction $\tilde h$ is locally $1$-Lipschitz continuous in the interior of $\tilde \Omega$ if $h$ is locally $1$-Lipschitz continuous in the interior of $\Omega$. Hence, it suffices to prove the first statement of the proposition for compact domains in $\R^n$. 

With some additional estimates, it can be verified that the second part of the proposition still holds if $N$ has bounded geometry, with a $c_n$ that also depends on $N$.

(ii) The exponent $n+1$ in the statement is optimal. Indeed, by the discussion in (i), there is some dimensional constant $\tilde c_n>0$ such that a $1$-Lipschitz function of the form $$k(x)\coloneqq k_{x_0,h_0}(x)\coloneqq \max \{0, h_0-d_\HH(x,x_0)\}$$ as in \eqref{eq:kfct} 
with $x\in \HH^n$ and $h_0\in (0,1/2)$  satisfies $\|k\|_{L^1(\HH^n)} < \tilde c_n h_0^{n+1}=\tilde c_n\|k\|_{L^\infty(\HH^n)}^{n+1}$.

(iii) Both the boundedness assumption and the Lipschitz assumption on the domain are necessary, and the $C^1$-regularity assumption is optimal; cf.~Subsection 
\ref{subsec:domain_optimality} for a similar discussion for Theorem \ref{thm:C^0_stability}. However, as can be seen from the proof of the proposition, it also holds for 
functions $h$ on $\Omega=\HH^n$ with $h_\Omega=\infty$.
\end{remsprp}

For the more technical proofs that are about to follow, we need the notion of a \emph{conical body} in $N\in \{\HH^n,\R^n\}$ with tip $x\in N$, radius $r>0$, and width $\theta \in (0,2)$ in the direction $e_x\in S_xN\coloneqq \{v\in T_xN \mid |v|=1\}$, which is defined as the convex set 
$$C_r(x,e_x,\theta)\coloneqq  \{\exp_x(tv)\mid 0\leq t\leq r, 2\sqrt{1-\langle e_x,v\rangle^2}\leq \theta, v\in S_x N, \langle e_x,v\rangle\geq 0\}\,,$$  where $\langle \cdot,\cdot\rangle$ is the Riemannian metric on $N$ with associated norm $|\cdot |$. Note that when regarding $\R^n=T_{x_0} \HH^n$, $x_0\in \HH^n$, this notation is consistent with the restriction of the Minkowski inner product also in case of $\R^n$.

To prove Proposition \ref{prp:L1_Linfty_bound_Lipschitz_smooth}, we make use of the following property: A set $\Omega\subset N$ satisfies the \emph{uniform interior cone condition} if there is some $r_\Omega \in (0,1)$ and some $\theta_\Omega \in (0,2)$ such that every point $x\in \Omega$ is the tip of a conical body of radius $r_\Omega$ and width $\theta_\Omega$ that is completely contained in~$\Omega$. While we prove the first statement of the proposition by relying on the fact that compact (Lipschitz) domains in $\R^n$ satisfy the uniform interior cone condition (see, for instance, \cite[Theorem 1.2.2.2]{Gr85}), the proof of the second part is based on the fact that for domains in~$\HH^n$ with $C^1$-boundary 
this holds with a universal width. 

\begin{lem}\label{lem:uniform_in-radius}
Let $\theta_0\in (0,2)$. Let $\Omega \subset \HH^n$ be a compact domain with $C^1$-boundary. Then $\Omega$ satisfies the uniform interior cone condition with $\theta_\Omega=\theta_0$.

\end{lem}
\begin{proof} The $C^1$-boundary gives a continuous inward unit normal vector field $\eta$ for $\partial \Omega \subset \HH^n$. We first show the following preliminary claim: There is some $r_\Omega>0$ such that $C_{2r_\Omega}(x,\eta_x,\theta_0)\subset \Omega$ for all $x\in \partial \Omega$.

 For $x,y\in \HH^n$ with $x\neq y$, we denote by $v_{x,y}\in T_x \HH^n$ the unit tangent vector pointing in the direction of $\exp^{-1}_x(y)$. If the preliminary claim were false, then we could find two sequences of points $x_j,y_j \in \partial \Omega$ with $x_j\neq y_j$ and $\langle \eta_{x_j},v_{x_j,y_j}\rangle \geq \delta\coloneqq\sqrt{1-(\theta_0/2)^2}>0$ for all $j\in \N$ and with $d_\HH(x_j,y_j) \rightarrow 0$ as $j\to\infty$. We abbreviate $v_j\coloneqq v_{x_j,y_j}$. After passing to subsequences, we can assume that $x_j,y_j\rightarrow x$ and $v_j \rightarrow v \in T_x \HH^n$ as $j \rightarrow \infty$. Then $\langle \eta_{x},v\rangle \geq \delta$.

We now regard $\HH^n$ as the unit hyperboloid in $\Lo^{n+1}$, so that we can think of $\partial \Omega \subset \HH^n$ as a $C^1$-submanifold of $\Lo^{n+1}$. For each $j$ the vector $y_j-x_j$ is spacelike and we can write
\[
    w_j \coloneqq\frac{y_j-x_j}{|y_j-x_j|}  = \lambda_j v_j + \mu_j x_j
\]
for some $\lambda_j, \mu_j \in \R$. By construction, the sequence $(w_j)_{j\in\N}$ stays in a compact set. Hence, we can assume that $w_j \rightarrow w \neq 0$ (after passing to another subsequence). Since $\partial \Omega$ is a $C^1$-submanifold, we have $w\in T_x \partial \Omega \subset T_x \HH^n$, so $\langle \eta_x,w\rangle =0$. This implies $\mu_j \rightarrow 0$ and $\lambda_j \rightarrow 1$ as $j \rightarrow \infty$ since both $v$ and $w$ are unit vectors. Therefore, $v=w \in T_x \partial \Omega$, which contradicts $\langle \eta_x,v\rangle >0$. Hence, the preliminary claim holds. 

Turning to the proof of the lemma, we first observe that for $y\in \Omega$ the point  $x\in \partial\Omega$ of minimal distance to $y$ satisfies $y=\exp_x(d_\HH(x,y)\eta_x)$. If $d_\HH(y,\partial\Omega)<r_\Omega$, then $$C_{r_\Omega}\left(y,\partial_t|_{t=d_\HH(x,y)} \exp(t\eta_x),\theta_0\right) \subset C_{2r_\Omega}(x,\eta_x,\theta_0)\subset \Omega$$
by elementary hyperbolic geometry. Since every conical body of radius $r_\Omega$ centered at a point in $\{y\in\Omega\mid d_\HH(y,\partial\Omega)\geq r_\Omega\}$ is completely contained in $\Omega$, the claim follows. 
\end{proof}

As $\HH^n$ is fully rotationally symmetric about 
every $x\in \HH^n$, we observe that for $r>0$, $\theta \in (0,2)$, $e_x \in S_x\HH^n$, and $e_0\in \mathbb S^{n-1}$  the ratio
\begin{equation}\label{eq:area_ratio}\frac{\mathrm{Area_{\HH^n}}(\partial  B_{r}(x) \cap C_r(x,e_x,\theta))}{\mathrm{Area_{\HH^n}}(\partial  B_{r}(x))} = \frac{\mathrm{Area_{\R^n}}(\partial  B^{\mathrm{Eucl}}_{r}(0) \cap C_r(0,e_0,\theta))}{\mathrm{Area_{\R^n}}(\partial  B^{\mathrm{Eucl}}_{r}(0))}\eqqcolon \alpha_n(\theta)
\end{equation}
is independent of $x$, $r$, $e_x$ and $e_0$. Here $B_r^{Eucl}(0)$ denotes the open $n$-dimensional Euclidean ball of radius $r>0$ and center $0$ as in the introduction, and $\mathrm{Area}_{N}$ denotes the $(n-1)$-dimensional Hausdorff measure on $N\in \{\R^n,\HH^n\}$.

\begin{proof}[Proof of Proposition \ref{prp:L1_Linfty_bound_Lipschitz_smooth}]
By Remarks on Proposition \ref{prp:L1_Linfty_bound_Lipschitz_smooth}, (i), we can assume for the first claim that $\Omega$ is a compact domain in $\R^n$.  In this case $\Omega$ satisfies the uniform interior cone condition with some radius $r_\Omega\in (0,1)$ and some width $\theta=\theta_\Omega \in (0,2)$ by \cite[Theorem 1.2.2.2]{Gr85}. For the second claim we work with a compact domain $\Omega$ in $\HH^n$ with $C^1$-boundary, for which the uniform interior cone condition is satisfied for some radius $r_\Omega\in (0,1)$ and width $\theta=1$ by Lemma~\ref{lem:uniform_in-radius}. The following discussion applies to both cases with $N\in \{\R^n,\HH^n\}$.

We define
\[
c_n(\theta)\coloneqq \frac{(n+1)}{\alpha_n(\theta)\mathrm{Vol}(B_1^{Eucl})} > 0
\]
with $\alpha_n(\theta)>0$ as defined in \eqref{eq:area_ratio}. Let $h:\Omega \rightarrow \R$ be a continuous function that is locally $1$-Lipschitz continuous in the interior of $\Omega$ with $\|h\|_{L^\infty(\Omega)} \leq r_\Omega$. The claimed inequality is trivial for $h$ vanishing identically. By \eqref{eq:localize_Linf_L1}, we can assume that $h\not \equiv 0$ is nonnegative and that $\Omega$ coincides with a conical body $C$ with tip $x_0\in \Omega$, radius $r_\Omega$, and width $\theta$, where $h(x_0) = \|h\|_{L^\infty(\Omega)}>0$.

In this case, by convexity of $C$ and continuity of $h$, the function $h$ is $1$-Lipschitz continuous with respect to $d$, which is the Euclidean metric $d_{Eucl}$ or $d_\HH$, respectively; cf.~Subsection \ref{subsec:induced_metrics}. Therefore, on~$\Omega$ the function $k:N\rightarrow \R$ defined by \begin{equation}\label{eq:kfct}
     k(x) \coloneqq \max \{0, h(x_0)-d(x,x_0)\}\,,\qquad x\in N\,,
 \end{equation} bounds $h$ from below. The function $k$ is radially symmetric with respect to $x_0$, bounded by $h(x_0)$, and supported on a ball of radius $h(x_0)$ 
 centered at $x_0\in C$. It follows that \begin{equation}
\label{eq:cone_ratio_bound}\alpha_n(\theta) 
 \|k\|_{L^1(N)}=  
 \|k\|_{L^1(C)}\leq \|h\|_{L^1(C)}\,. \end{equation} 
 We additionally need the estimate 
\begin{equation}
\label{eq:EuclHyp}\frac{\mathrm{Vol}(B_1^{Eucl})}{n+1} |h(x_0)|^{n+1}= \frac{\mathrm{Vol}(B^{Eucl}_{|h(x_0)|}) |h(x_0)|}{n+1}  \leq\|k\|_{L^1(N)} \, ,
\end{equation}	
which is an equality for $N=\R^n$ and which holds for $N=\HH^n$ as the volume of a Euclidean $(n-1)$-sphere of radius $r$ 
is bounded from above by the volume of a hyperbolic $(n-1)$-sphere of radius $r$. Hence, by \eqref{eq:cone_ratio_bound} and \eqref{eq:EuclHyp}, we deduce 
\[
			\|h\|^{n+1}_{L^\infty(C)} \leq c_n(\theta) \|h\|_{L^1(C)}\, .
\]

This concludes the first part of the proposition with $c_\Omega=c_n(\theta)$ and $h_\Omega=r_\Omega$ on $\R^n$ (and thus on $\HH^n$ up to an additional bi-Lipschitz constant). If $\Omega$ has a $C^1$-boundary, then $c_n(\theta)=c_n(1)$ is independent of $\Omega$, and thus also the second part of the proposition is proved with $c_\Omega=c_n(1)$. \end{proof}

To be able to apply Proposition \ref{prp:L1_Linfty_bound_Lipschitz_smooth}, we in addition need the following a priori bound on $\|h\|_{L^\infty(\Omega)}$ in terms of $\|h\|_{L^1(\Omega)}$. Here ``a priori'' is to be understood in the sense that in contrast  to Proposition~\ref{prp:L1_Linfty_bound_Lipschitz_smooth} we do not assume any smallness on $h$ beforehand.

\begin{lem}\label{lem:l1_to_sup}  
Let $\Omega \subset \HH^n$ be a compact (Lipschitz) domain.
Then there is a constant $c'_{\Omega}>0$ such that for every continuous function $h:\Omega \to \R$ that is locally $1$-Lipschitz continuous in the interior of $\Omega$, it holds that
\[
    \|h\|_{L^\infty(\Omega)}\min\{1,\|h\|_{L^\infty(\Omega)} \}^{n} \leq c'_{\Omega} \|h\|_{L^1(\Omega)}\,.
\]
\end{lem}

By Remarks on Proposition \ref{prp:L1_Linfty_bound_Lipschitz_smooth}, (i), this lemma also holds verbatim for general Riemannian manifolds, and it suffices to prove it for compact (Lipschitz) domains in $\R^n$.

\begin{proof} As pointed out above, we can equivalently prove the statement for compact domains $\Omega$ in $\R^n$. The statement follows immediately from the proof of the first part of Proposition~\ref{prp:L1_Linfty_bound_Lipschitz_smooth} if $\|h\|_{L^\infty(\Omega)}\leq h_\Omega$. 
Thus, assume that $\|h\|_{L^\infty(\Omega)}> h_\Omega$. 

By \cite[Theorem 1.2.2.2]{Gr85}, $\Omega$ satisfies the uniform interior cone condition with radius $r_\Omega\in (0,\min\{1,h_\Omega\})$ and width $\theta_\Omega \in (0,2)$. 
In particular, $x_0\in \Omega$ with $|h(x_0)|=\|h\|_{L^\infty(\Omega)}>r_\Omega$ is the tip of a conical body of radius $r_\Omega$ contained in $\Omega$. As $h$ is locally $1$-Lipschitz continuous, we have $|h| \geq \|h\|_{L^\infty(\Omega)}/2$ on a conical body $C_\Omega\subset \Omega$ of radius $r_{\Omega}/2$ so that 
$$ \frac{\|h\|_{L^\infty(\Omega)}}{2} \mathrm{Vol}(C_\Omega)\leq \|h\|_{L^1(\Omega)}\,.$$ Hence, the claim holds on $\R^n$ for $c'_{\Omega}= \max \{c_\Omega, 2\mathrm{Vol}(C_\Omega)^{-1}\}$ and $c_\Omega=c_n(\theta_\Omega)$ as in the proof of Proposition \ref{prp:L1_Linfty_bound_Lipschitz_smooth} (and thus on $\HH^n$ up to an additional bi-Lipschitz constant). \end{proof}

\begin{rem} 
The constants $c_\Omega$ and $h_\Omega$ in Proposition~\ref{prp:L1_Linfty_bound_Lipschitz_smooth} and $c'_\Omega$ in Lemma~\ref{lem:l1_to_sup} have to depend on the domain $\Omega$. Looking at larger and larger constant functions on a closed ball of radius $r$ shows that a bound on $h$ is necessary in Proposition \ref{prp:L1_Linfty_bound_Lipschitz_smooth} by scaling. Sending $r\rightarrow 0$ shows that this bound has to depend on $\Omega$ in view of $\|h\|_{L^1(\Omega)} \leq \mu(\Omega) \|h\|_{L^{\infty}(\Omega)}$. The dependence of $c_\Omega$ and $c_\Omega'$ on $\Omega$ becomes apparent by looking at functions $k_{x_0,h_0}(x)$, $h_0\in (0,1)$, as in Remarks on Proposition \ref{prp:L1_Linfty_bound_Lipschitz_smooth}, (ii), restricted to conical bodies $C_1(x_0,e_1,\theta)$ with $\theta \rightarrow 0$; cf.~proof of Claim 4 in Subsection~\ref{subsec:domain_optimality}. 
\end{rem}

We can finally prove the main result of this section.

\subsection{Proof of Theorem \ref{thm:C^0_stability}}

We first give some preliminary observations. By scale invariance, we can assume that $$d_J(O,S)^{n+1} = \sup_{v\in S} |v|^{n+1} = \frac{1}{n+1}\,.$$ We write $S=S_f$ for some function $f:\Omega=\pi(S)\to \R_+$ as in Lemma \ref{lem:graph}, where $f$ is locally Lipschitz and $\ln(f)$ is $1$-Lipschitz continuous with respect to $d_\Omega$. We set $g=f^{n+1}$ so that $V(C(S))= ||g||_{L^1(\Omega)}/(n+1)$. 
Note that $g$ is locally $1$-Lipschitz continuous with respect to $d_\Omega$ as $$ |g(x)-g(y)|\leq (n+1)\|f\|^n_{L^\infty(\Omega)} |f(x)-f(y)|\leq  (n+1)d_J(O,S)^{n+1} d_\Omega(x,y)= d_\Omega(x,y)\,,$$ for every $x,y\in \Omega$. Here we used Lemma \ref{lem:graph}, (i), in the form $\|\nabla f\|_{L^\infty(\Omega)}\leq \|f\|_{L^\infty(\Omega)}=d_J(O,S)$, the fact that $f$ is $d_J(O,S)$-Lipschitz continuous with respect to $d_\Omega$ (see Subsection~\ref{subsec:induced_metrics}), and the previous normalization. As we see next, this property of $g$ is essential to exploit the previous lemmas.

We prove both parts of the theorem simultaneously. Assume first that $A_F(S)\leq \delta_M$ with \begin{equation}\label{eq:choicedelM}
  \delta_M\coloneqq \frac 1 {c'_{\Omega}V(B_1(M))}\left(\min \left\{ h_\Omega,\frac 1{4(n+1)}\right\}\right)^{n+1}\,,
\end{equation} where
$h_\Omega$ is given as in Proposition~\ref{prp:L1_Linfty_bound_Lipschitz_smooth} and $c_\Omega'$ as in Lemma~\ref{lem:l1_to_sup}.
Let us write $d^*= d_J(O,S)$ and $d_*= \inf |S|$. Choose $t\in [d_*,d^*]$ so that $ \|g-t^{n+1}\|_{L^1(\Omega)} = (n+1)V(C(S)) A_F(S)$.
Since \begin{equation}
\label{eq:volbound}(n+1)V(C(S)) \leq (d^*)^{n+1} \mu(\Omega) =  V(B_1(M))\,,
\end{equation} 
 we find
$\|g-t^{n+1}\|_{L^1(\Omega)} 
\leq V(B_1(M)) \delta_M 
$ by assumption.
Then Lemma \ref{lem:l1_to_sup} implies \begin{equation}\label{eq:r_Omega_n_bound}
    \|g-t^{n+1}\|_{L^\infty(\Omega)}\leq \min \left\{ h_\Omega,\frac 1{4(n+1)}\right\}.
\end{equation}
 The first part of Proposition \ref{prp:L1_Linfty_bound_Lipschitz_smooth} now yields
\begin{equation}\label{eq:Linftydelta}
    \|g-t^{n+1}\|_{L^\infty(\Omega)}^{n+1} \leq c_\Omega \|g-t^{n+1}\|_{L^1(\Omega)}\,.
\end{equation}
 Moreover, because of $(d^*)^{n+1}-(d_*)^{n+1} \leq 2 \|g-t^{n+1}\|_{L^\infty(\Omega)}$, the dimensional upper bound in \eqref{eq:r_Omega_n_bound} gives $1 \leq 2(n+1)d_*^{n+1}$. Using this bound, we can estimate
\begin{equation*}
    \|f-t\|_{L^\infty(\Omega)}
\leq
(2(n+1))^{\frac{n}{n+1}}d_*^{n} \|f-t\|_{L^\infty(\Omega)}
         \leq \left(\frac{2^n}{n+1}\right)^{\frac{1}{n+1}}\|g-t^{n+1}\|_{L^\infty(\Omega)} \, .
\end{equation*}
Combining the latter estimate with \eqref{eq:AH_funct}, \eqref{eq:Linftydelta}, the definition of $t$, 
and \eqref{eq:volbound} now leads to
\begin{equation}\label{eq:AHAFfinal}
	A_{H}(S)^{n+1} \leq (n+1) \|f-t\|_{L^\infty(\Omega)}^{n+1} 
    \leq 2^n c_\Omega \|g-t^{n+1}\|_{L^1(\Omega)}\leq 2^n c_\Omega V(B_1(M))A_F(S)\,.
\end{equation}
 This proves the first part of the theorem for $c_M=2^nc_\Omega V(B_1(M))$ and $A_F(S)\leq \delta_M$. 
 If $A_F(S)> \delta_M$, we immediately find a bound with $c_M=(2^{n+1}\delta_M)^{-1}$ using $A_H(S)\leq 1/2$.
 
For the second part of the theorem, we assume $M$ to be $C^1$-regular.

By \eqref{eq:r_Omega_n_bound} we can then apply the second part of Proposition~\ref{prp:L1_Linfty_bound_Lipschitz_smooth} so that $c_\Omega$ in \eqref{eq:Linftydelta} is a dimensional constant. If $A_F(S)\leq \delta_M$, this leads to a constant $c_M=c_nV(B_1(M))$ in \eqref{eq:AHAFfinal}, as claimed.
 
Sharpness of the exponent and domain dependence is discussed in Subsection \ref{subsec:Linf_optimality} and Subsection
\ref{subsec:domain_optimality}, respectively.
\hfill\qed

\subsection{Sharpness of the exponent} 
\label{subsec:Linf_optimality}
In this subsection we show that the exponent $n+1$ in Theorem \ref{thm:C^0_stability} is sharp, that is, there is a family $S_\varepsilon$ of Cauchy hypersurfaces of a \emph{fixed} smooth 
spatially compact conical Minkowski spacetime such that
\begin{equation}\label{eq:optdim1}
   \lim_{\varepsilon\to 0}A_H(S_{\varepsilon})= 0 \quad \text{and} \quad \limsup_{\varepsilon\to 0} \frac{A_F(S_\varepsilon)}{A_H(S_\varepsilon)^{n+1}}<\infty\,.
\end{equation} 
Based on the construction in Remarks on Proposition \ref{prp:L1_Linfty_bound_Lipschitz_smooth}, (ii), we define\begin{equation} \label{eq:kappa_eps}\kappa_{\varepsilon}\coloneqq 1+k_{x_0,\varepsilon} 
\end{equation} with $\varepsilon<1/2$ and $x_0\in \HH^n$. This function describes a Cauchy hypersurface $S_{\varepsilon}=S_{\kappa_{\varepsilon}}$ of the smooth spatially compact conical Minkowski spacetime $M=\R_+ \bar B_1(x_0)$. Indeed, as $\ln$ on $[1,\infty)$ and $k_{x_0,\varepsilon}$ are $1$-Lipschitz continuous, $\ln(\kappa_\varepsilon)$ is $1$-Lipschitz continuous, which implies that $S_{\kappa_\varepsilon}$ is a Cauchy hypersurface of $M$ by Lemma~\ref{lem:cauchy_graph}. Note that $d_\Omega$ and $d_\HH$ agree again as the domain is convex. Clearly, $A_H(S_\varepsilon)= \varepsilon (2(1+\varepsilon))^{-1}\in( \varepsilon/3, \varepsilon)$ as $\varepsilon<1/2$. Taking $t=1$ as competitor in~\eqref{eq:AF_tilde}, we observe that
\begin{equation}
    \label{eq:AH_AF}A_F(S_\varepsilon)\leq \frac{2^{n+1}\|k_{x_0,\varepsilon}\|_{L^1(\HH^n)}}{V(B_1(M))}\leq  \frac{c_n'\varepsilon^{n+1}}{V(B_1(M))}
\end{equation}
for some dimensional constant $c_n'>0$, where we used Lemma \ref{lem:asymmetry_comparsion} and the elementary inequality $(1+a)^{n+1}-1\leq (n+1) 2^n a$ for $0\leq a\leq 1$. This shows \eqref{eq:optdim1}. Since $A_F(S_\varepsilon)$ can be chosen to be smaller than a given $\delta_M$, sharpness of the exponent $n+1$ follows as claimed. 

We note that families of the form $\kappa_\varepsilon$ do not establish sharpness for the Hausdorff stability results in Corollary \ref{cor:C0}. Indeed, they only show that the exponent of the Hausdorff asymmetry needs to be at least  $n$ in case of the CM-inequality and at least $n/2$ in case of the BE-inequality and the refined CM-inequality.

\subsection{Optimal domain dependence} 
\label{subsec:domain_optimality}
In this subsection we study optimal dependence on the domain in four different respects. As they always involve the boundary of the domain, they only make sense for $n\geq 2$. Our first two objectives are to show that the boundedness and Lipschitz regularity assumptions on the domain in Theorem \ref{thm:C^0_stability} are necessary.
\vspace{0.2cm}

\noindent\textit{Claim 1.}~The assumption that the domain has a Lipschitz boundary is sharp, that is, the first part of Theorem
\ref{thm:C^0_stability} fails for Cauchy hypersurfaces over H\"older domains, i.e.~generalized domains in $\HH^n$ with mere $C^{0,\alpha}$-boundary, $0<\alpha<1$. This means that there is a family $S_\varepsilon$ of Cauchy hypersurfaces of a \emph{fixed} spatially compact conical Minkowski spacetime over a H\"older domain such that 
\begin{equation}\label{eq:optdim2}
   \limsup_{\varepsilon\to 0}A_H(S_{\varepsilon})<\infty \quad \text{and} \quad \lim_{\varepsilon\to 0} \frac{A_F(S_\varepsilon)}{A_H(S_\varepsilon)^{n+1}}=0\,.
\end{equation}

\begin{proof}[Proof of Claim 1] 
   Consider for the domain $\Omega$ the closure of a unit ball $ B_1(x_1)$, $x_1\in \HH^n$, with an attached outward $C^{0,\alpha}$-cusp, $0<\alpha<1$, with tip  at $x_0\in \partial B_2(x_1)$. We further consider the Cauchy hypersurfaces $S_\varepsilon$ of $\R_+\Omega$ described by $\kappa_{\varepsilon}$ as in \eqref{eq:kappa_eps}. Then for every $\varepsilon \in (0,1/2)$ there is some $r>0$ and some $e_1\in T_{x_0} \HH^n$ such that $\Omega \cap B_r(x_0) \subset C_r(x_0,e_1,\varepsilon)$. Hence, we find by \eqref{eq:area_ratio} that
$$ \|k_{x_0,\varepsilon}\|_{L^1(\Omega)}\leq \|k_{x_0,\varepsilon}\|_{L^1(C_r(x_0,e_1,\varepsilon))}=\alpha_n(\varepsilon)\|k_{x_0,\varepsilon}\|_{L^1(\HH^n)}\,. 
$$ 
 Then $A_H(S_\varepsilon)\in( \varepsilon/3, \varepsilon)$ as in Subsection \ref{subsec:Linf_optimality}, and by the argument leading to \eqref{eq:AH_AF} we have \begin{equation}\label{eq:AF_Calpha}
A_F(S_\varepsilon)\leq
     \frac{ c_n'\alpha_n(\varepsilon)\varepsilon^{n+1}}{V(B_1(\R_+S_\varepsilon))} \,.
    \end{equation} As $V(B_1(\R_+S_\varepsilon))\geq V(B_1(\R_+B_1))$ and $\alpha_n(\varepsilon)\to 0$ as $\varepsilon\to0$, this concludes \eqref{eq:optdim2}. 
    \end{proof}

\noindent\textit{Claim 2.}~The assumption that the domain is bounded is necessary, that is, there is a family $S_\varepsilon$ of Cauchy hypersurfaces of a \emph{fixed} smooth conical Minkowski spacetime over an unbounded domain of finite volume such that \eqref{eq:optdim2} holds.

\begin{proof}[Proof of Claim 2] 
     Consider for the domain $\Omega$ the closure of the unit ball $ B_1(x_0)$ in $\HH^n$ with an unbounded tentacle of finite volume. By assumption, we must have $\|k_{y,\varepsilon}\|_{L^1(\Omega)}\to 0$ for fixed $\varepsilon>0$ and $y\in \Omega$ traveling along a tentacle. Hence, we can choose $y_{\varepsilon}\in \Omega$ such that $$ \|k_{y_{\varepsilon},\varepsilon}\|_{L^1(\Omega)}\leq \varepsilon \|k_{y_{\varepsilon},\varepsilon}\|_{L^1(\HH^n)}\,.$$ 
     
     We further consider Cauchy hypersurfaces $S_{\varepsilon}$ of $\R_+\Omega$ described by $\kappa_{\varepsilon}$ as in \eqref{eq:kappa_eps} with $x_0=y_{\varepsilon}$. By construction, we find $A_H(S_{\varepsilon})\in( \varepsilon/3, \varepsilon)$ and $$A_F(S_{\varepsilon})\leq \frac{c_n'\varepsilon^{n+2}} {V(B_1(\R_+\Omega))}$$
     which concludes \eqref{eq:optdim2}. Here we used the estimates leading to \eqref{eq:AH_AF} again.
   \end{proof}

The other two goals are concerned with the $C^1$-regular setting, that is, the second part of Theorem~\ref{thm:C^0_stability}. 
Recall that in this setting we have to assume $A_F(S)$ to be smaller than $\delta_M$ as given in \eqref{eq:choicedelM}. Our third objective is to show that the dependence of $c_M$ on $V(B_1(M))$ is optimal. 
\vspace{0.2cm}

\noindent\textit{Claim 3.}~The dependence of $c_M$ on $V(B_1(M))$ in Theorem~\ref{thm:C^0_stability} is sharp, that is, there is a family $S_\varepsilon$ of Cauchy hypersurfaces of \emph{non-fixed} smooth spatially compact conical Minkowski spacetimes such that $A_F(S_{\varepsilon})\leq \delta_{\R_+S_\varepsilon}$ for all $\varepsilon$ sufficiently small,\begin{equation}
    \label{eq:optdim3}
\lim_{\varepsilon\to 0}A_H(S_{\varepsilon})=0\,,\quad \text{and} \quad\limsup_{\varepsilon\to0}\frac{A_F(S_\varepsilon)}{A_H(S_\varepsilon)^{n+1}}V(B_1(\R_+S_\varepsilon))<\infty\,.\end{equation}

\begin{proof}[Proof of Claim 3]   Consider the Cauchy hypersurfaces $S_\varepsilon$ of  $\R_+\bar B_{1/\varepsilon}(x_0)$ described by $\kappa_{\varepsilon}$ as in~\eqref{eq:kappa_eps}. Similarly to Subsection \ref{subsec:Linf_optimality}, we observe that $A_H(S_\varepsilon)\in( \varepsilon/3, \varepsilon)$ and $$A_F(S_\varepsilon)\leq 
     \frac{ c_n'\varepsilon^{n+1}}{V(B_1(\R_+S_\varepsilon))}\,,$$ and thus \eqref{eq:optdim3} holds. Note that $\delta_{\R_+S_\varepsilon}=\delta_{\R_+\bar B_{1/\varepsilon}(x_0)}$ is (up to a dimensional factor) bounded from below by $V(B_1(\R_+\bar B_{1/\varepsilon}(x_0)))^{-1}=V(B_1(\R_+S_\varepsilon))^{-1}$; cf.~proof of Corollary \ref{cor:asymptotic_stability} in Subsection~\ref{subsec:asymptotic}. 
     Hence, by the previous estimate, $A_F(S_{\varepsilon})\leq \delta_{\R_+S_\varepsilon}$ for all $\varepsilon$ sufficiently small. 
\end{proof}

Our fourth objective is to show that the additional regularity assumption on the boundary of the domain is necessary.
\vspace{0.2cm}

\noindent\textit{Claim 4.}~The assumption that the domain has a $C^1$-boundary is sharp, that is, the addendum for $C^1$-regular $M$ in Theorem \ref{thm:C^0_stability} fails for (Lipschitz) domains. This means that there is a family $S_\varepsilon$ of Cauchy hypersurfaces of \emph{non-fixed}  spatially compact conical Minkowski spacetimes over (Lipschitz) domains such that $A_F(S_{\varepsilon})\leq \delta_{\R_+S_\varepsilon}$ for all $\varepsilon$ sufficiently small,\begin{equation}
    \label{eq:optdim4}
 \limsup_{\varepsilon\to 0}A_H(S_{\varepsilon})<\infty\,, \quad \text{and}\quad \lim_{\varepsilon\to0}\frac{A_F(S_\varepsilon)}{A_H(S_\varepsilon)^{n+1}}V(B_1(\R_+S_\varepsilon))=0\,.\end{equation} 

\begin{proof}[Proof of Claim 4]  Consider the Cauchy hypersurfaces $S_\varepsilon$ of $\R_+\Omega_\varepsilon$ described by $\kappa_{\varepsilon}$ as in \eqref{eq:kappa_eps}, where $\Omega_\varepsilon= C_1(x_0,e_1,\varepsilon)$ is  
defined as in Subsection \ref{subsec:prel_bdd}.  
This gives $A_H(S_\varepsilon)\in( \varepsilon/3, \varepsilon)$ and \eqref{eq:AF_Calpha}, which implies \eqref{eq:optdim4}. Observe that $\delta_{\R_+S_\varepsilon}=\delta_{\R_+\Omega_\varepsilon}$ as given in \eqref{eq:choicedelM} is uniformly bounded from below in $\varepsilon$ in this construction. Indeed, the uniform interior cone condition holds for cones $\Omega_\varepsilon'= \exp_{x_0}^{-1}(C_1(x_0,e_1,\varepsilon))$ in $\R^n$ with uniform $r_{\Omega_\varepsilon'}=1/3$ and $\theta_{\Omega_\varepsilon'}=\varepsilon$ for sufficiently small $\varepsilon$ by simple geometric considerations. Since the differential of the diffeomorphism $\exp_{x_0}: T_{x_0} \HH^n \rightarrow \HH^n$ is uniformly continuous on the unit ball bundle of $T_{x_0} \HH^n$, the images of those uniform interior cones contain geodesic cones in $\HH^n$ with a uniform radius and a width that is proportional to $\varepsilon$. Hence, $V(C_{\Omega_\varepsilon})$ from the proof of Lemma~\ref{lem:l1_to_sup} and $V(B_1(\R_+\Omega_\varepsilon))$ scale the same in $\varepsilon$, so $\delta_{\R_+S_\varepsilon}$ is bounded from below uniformly in $\varepsilon$.
\end{proof}

\subsection{Quantitative stability for the Milne universe}\label{subsec:asymptotic} 

A simple cosmological model with a big bang is the Milne universe $M=I^+(O)=\R_+ \HH^n$; see \cite{Mi35}. Although this globally hyperbolic conical Minkowski spacetime is not defined over a domain of finite measure, we can still formulate an asymptotic stability result for it.

For a Cauchy hypersurface $S$ of $M$, we consider an exhaustion of $S$ by compact achronal Lipschitz hypersurfaces $S^{(j)}\coloneqq S\cap M^{(j)}$, $j\in \N$, in conical Minkowski spacetimes $M^{(j)}\coloneqq\R_+ \bar B_j(x_0)$ for some fixed base point $x_0\in \HH^n$. We set $$A_F(S) \coloneqq  \liminf_{j \to \infty}  A_F(S^{(j) })\qquad \text{and}\qquad \delta(S) \coloneqq \liminf_{j\to \infty} \delta(S^{(j)})$$ for $\delta\in\{\delta_{BE}^{1/2},\delta_{CM},(\delta_{CM}^*)^{1/2}\}$. Note that these definitions a priori depend on the chosen base point. Nevertheless, 
if $\mu(\pi(S)) < \infty$ and $V(C(S))<\infty$, 
the limits inferior $\delta(S)$ and $A_F(S)$ are proper limits and given by the usual formulas. With these definitions, we obtain natural extensions of Theorem \ref{thm:BE_L1_stability}, Corollary~\ref{cor:CM_L1_stability}, and Corollary~\ref{cor:CM_refined_L1_stability} to non-compact domains, which can be summarized as follows.

\begin{cor}\label{cor:asympt_Fr} There is an explicitly computable constant $c_n>0$ such that for every Cauchy hypersurface $S \subset I^+(O)$ we have
$A_F(S)
            \leq c_n \delta(S)$.
\end{cor}

However, neither $\delta(S)=0$ nor $A_F(S)=0$ imply that $S$ is contained in a hyperboloid anymore. To recover this rigidity, we work with the Hausdorff asymmetry $A_H$ and the rescaled deficit $\tilde \delta$, which we define as limits through
$$A_H(S) \coloneqq  \lim_{j \to \infty}  A_H(S^{(j) })\qquad \text{and}\qquad \tilde \delta(S) \coloneqq  \liminf_{j\to \infty}  V(B_1(M^{(j)})) \delta(S^{(j)})\,.
$$ The first limit always exists in $[0,\infty]$ since \begin{equation}\label{eq:AHreform}
    A_H(S^{(j)})= \frac 12\left(1-\frac {\inf|S^{(j)}|}{d_J(O,S^{(j)})}\right)
\end{equation} is increasing in $j$. In fact, $A_H$ could also be defined in the same way as in the compact case since the assumption to be spatially compact is not necessary for $d_J$ to be a metric; see \cite{LP26a}. Even in the non-compact case, $A_H(S)=0$ implies that $S$ is contained in a hyperboloid.

Note that in contrast to Corollary \ref{cor:C0} we do not require smallness of $\tilde \delta$ in Corollary \ref{cor:asymptotic_stability}. To drop this condition, we have to show that $V(B_1(M^{(j)}))\delta_{M(j)}$ is bounded from below uniformly in $j$.

\begin{proof}[Proof of Corollary \ref{cor:asymptotic_stability}] Recall the definition of $\delta_M$  in \eqref{eq:choicedelM}. For $\Omega=\bar B_j(x_0)$ and sufficiently large $j \in \N$, we can choose $r_{\Omega}\in (0,1)$ large enough and $\theta_\Omega=1$ in the uniform interior cone condition such that $h_\Omega\geq (4(n+1))^{-1}$, $c_n^*\coloneqq c_\Omega' \geq c_\Omega=c_n(1)$, and $\delta_{M^{(j)}}=(c_n^*V(B_1(M^{(j)}))(4(n+1))^{n+1})^{-1}$ with $M^{(j)}=\R_+\bar B_j(x_0)$; cf.~the proof of Proposition \ref{prp:L1_Linfty_bound_Lipschitz_smooth} and Lemma \ref{lem:l1_to_sup}. Note that the proof of Lemma \ref{lem:l1_to_sup} works on $\HH^n$ the same way here thanks to Lemma \ref{lem:uniform_in-radius}, so $c_n^*=c_\Omega'$ can be chosen to depend only on the dimension and not on the bi-Lipschitz constant.

By Theorem \ref{thm:BE_L1_stability}, Corollary \ref{cor:CM_L1_stability}, and Corollary \ref{cor:CM_refined_L1_stability}, there is a $\kappa_n>0$ such that $\delta(S^{(j)})<\kappa_n\delta_{M^{(j)}}$ implies $A_F(S^{(j)})\leq \delta_{M^{(j)}}$. If $\delta(S^{(j)})<\kappa_n\delta_{M^{(j)}}$, then by \eqref{eq:HF_iso} with
$c_{M^{(j)}}=c_nV(B_1(M^{(j)}))$ we have \begin{equation}\label{eq:est1}
			A_{H}(S^{(j)})^{n+1} \leq c_nV(B_1(M^{(j)}))\delta(S^{(j)}) \, .
\end{equation} 
If $\delta(S^{(j)})\geq \kappa_n
\delta_{M^{(j)}}$, then by \eqref{eq:AHreform} we know that $A_H(S^{(j)})\leq 1/2$, so
\begin{equation}\label{eq:est2}
    A_H(S^{(j)})^{n+1}\leq \frac {(4(n+1))^{n+1}c_n^*}{2^{n+1}\kappa_n} V(B_1(M^{(j)}))\delta(S^{(j)})\,.
\end{equation}
Combining \eqref{eq:est1} and \eqref{eq:est2} and taking $\liminf$ proves the claim as $(n+1)V(B_1(M^{(
j)}))=\mu(B_j(x_0))$.
\end{proof}

\appendix

\section{Complementary stability results}
\label{app:compl_pf} 

Here we compare our stability constant from Theorem \ref{thm:BE_L1_stability} with the stability constants obtained through quantitative H\"older inequalities from \cite{CFL14, Al08}. To this end, let $S$, $M$, $f$, $g$, $\bar g$, and $\phi$ be as used in Subsection~\ref{subsec:func_pf}. Assume that $ V(C(S)),V(C(\pi(S)))\in (0,\infty)$ holds similarly to the second proof of Theorem~\ref{thm:Bahn_Ehrlich_gen}.

Applying \cite[Theorem~3.1]{CFL14} for $(X, \mu)=(\pi(S),\mu)$, $p=n+1$, and the unit vector $(g/\|g\|_{L^{1}(\pi(S))})^{n/(n+1)}$ in $L^{(n+1)/n}(\pi(S))$ gives 
\begin{equation*}
\left|\int_{\pi(S)} \frac{1}{\mu(\pi(S))^{\frac 1 {n+1}}}\frac{g^{\frac n{n+1}}}{\|g\|^{\frac n{n+1}}_{L^{1}(\pi(S))}}\,\mathrm d\mu \right|\leq 1- \frac 1{\beta_1(r)}\left\|\frac{1}{\mu(\pi(S))^{\frac 1r}}-\left(\frac{g}{\|g\|_{L^{1}(\pi(S))}}\right)^{\frac 1r}\right\|_{L^{r}(\pi(S))}^{\max\{r,2\}}\,,
\end{equation*} for $r\in \{(n+1)/n,n+1\}$,
where
$\beta_1((n+1)/n)\coloneqq  4n$ and $\beta_1(n+1)\coloneqq (n+1)2^n$. Similarly, applying the second inequality of \cite[Theorem 2.2]{Al08} for $(X,\mu)=(\pi(S), \mu)$, $p=(n+1)/n$, and $g^{n/(n+1)}$ in $L^{(n+1)/n}(\pi(S))$ gives an estimate of the same form with $r=2$ and $\beta_1(2)\coloneqq(n+1)$.

We recall the definition of the Bahn--Ehrlich deficit $\delta_{BE}$ in \eqref{eq:BE_ineq}. Rearranging the terms yields
\begin{align}\notag
\frac 1{\beta_1(r)}\left\|\frac{1}{\mu(\pi(S))^{\frac 1r}}-\left(\frac{g}{\|g\|_{L^{1}(\pi(S))}}\right)^{\frac 1r}\right\|_{L^{r}(\pi(S))}^{\max\{r,2\}}&\leq 1-\frac{\int_{\pi(S)} g^{\frac n{n+1}}\,\mathrm d\mu}{\mu(\pi(S))^{\frac 1 {n+1}}\|g\|^{\frac n{n+1}}_{L^{1}(\pi(S))}} \\&\leq 1-\frac{1}{1+\delta_{BE}(S)}\leq \delta_{BE}(S)\,.\label{eq:prelBEstab1}
\end{align}

Note that the Fraenkel asymmetry 
$$A_F(S)=\left\|\frac{1}{\mu(\pi(S))}-\frac{g}{\|g\|_{L^{1}(\pi(S))}}\right\|_{L^{1}(\pi(S))}$$ is 
the distance on the left side of \eqref{eq:prelBEstab1} with $r=1$. Hence, it makes sense to denote the distance in \eqref{eq:prelBEstab1} by $$A_F^{(r)}(S)\coloneqq \left\|\frac{1}{\mu(\pi(S))^{\frac 1r}}-\left(\frac{g}{\|g\|_{L^{1}(\pi(S))}}\right)^{\frac 1r}\right\|_{L^{r}(\pi(S))}.$$ Moreover, this distance is on normalized functions at least as strong as the Fraenkel asymmetry. Indeed, using the simple consequence of the mean value theorem $$|a^r-b^r|\leq r(a^{r-1}+b^{r-1})|a-b| $$ with $a,b\geq 0$ and $ r>1$ and H\"older's inequality, it follows that
$$A_F(S)\leq2r A_F^{(r)}(S)\,,$$ and hence
$$A_F(S)^{\max\{r,2\}}\leq \beta_2(r)\delta_{BE}(S)$$
with $r\in \{(n+1)/n,n+1,2\}$, where \begin{equation*}
    \beta_2\left(\frac{n+1}n\right)\coloneqq 16\frac{(n+1)^2}{n} \,, \qquad\beta_2(n+1)\coloneqq(n+1)^{n+2}2^{2n+1}\,, \qquad\beta_2(2)\coloneqq 16(n+1)\,.
\end{equation*} 
For $r=(n+1)/n$ and $r=2$, the exponent of the Fraenkel asymmetry is optimal. 
These stability constants should be compared with the constant $2(n+1)^2/n$ in Theorem \ref{thm:BE_L1_stability}.
More generally, it holds that $$A_F^{(t)}(S)\leq 2\frac{r}{t}A_F^{(r)}(S)$$ for $t\leq r$, leading to some immediate stability results with geometry-independent constant with respect to an $L^t$-distance of $(g/\|g\|_{L^1(\pi(S))})^{1/t}$ for $t<n+1$.


\begin{thebibliography}{[DEFL23]}


\bibitem[ACKW09]{ACKW09} F.~Abedin, J.~Corvino, S.~Kapita, and H.~Wu, \textit{On isoperimetric surfaces in general relativity, II}, J. Geom. Phys. \textbf{59} (2009), no.~11, 1453--1460.

\bibitem[Al08]{Al08} J.~M.~Aldaz, \textit{A stability version of Hölder's inequality}, J. Math. Anal. Appl. \textbf{343} (2008), no. 2, 842--852.


\bibitem[AGKS23]{AGKS23} S.~B.~Alexander, M.~Graf, M.~Kunzinger and C.~Sämann, \textit{Generalized cones as Lorentzian length spaces: causality, curvature, and singularity theorems}, Comm. Anal. Geom. \textbf{31} (2023), no. 6, 1469--1528.

\bibitem[BE99]{BE99}  H.~Bahn and P.~Ehrlich, \textit{A Brunn-Minkowski type theorem on the Minkowski spacetime}, Canad. J. Math. \textbf{51} (1999), no. 3, 449--469.

\bibitem[Ba99]{Ba99}  H.~ Bahn, \textit{Isoperimetric inequalities and conjugate points on Lorentzian surfaces}, J. Geom. \textbf{65} (1999), no.~1-2, 31--49.

\bibitem[BB61]{BB61} E.~Beckenbach and R.~Bellman, \textit{Inequalities}, Springer-Verlag, Berlin (1961).


\bibitem[Be05]{Be05} F.~Bernstein, \textit{Über die isoperimetrische Eigenschaft des Kreises auf der Kugeloberfl\"ache und in der Ebene}, Math. Ann. \textbf{60} (1905), 117--136.


\bibitem[BE91]{BE91}
	G.~Bianchi and H.~Egnell, \emph{A note on the Sobolev inequality}, J.
	Funct. Anal. \textbf{100} (1991), no.~1, 18--24.

\bibitem[Bo24]{Bo24} T. Bonnesen, \textit{\"Uber das isoperimetrische Defizit ebener Figuren}, Math. Ann. \textbf{91} (1924), 252--268.

\bibitem[BC04]{BC04} H.~L.~Bray and P.~T.~Chruściel, \textit{The Penrose Inequality}, In:
The Einstein Equations and the Large Scale Behavior of Gravitational Fields. 50 Years of the Cauchy Problem
in General Relativity, Birkhäuser, Basel (2004).

\bibitem[BE13]{BE13} S.~Brendle and M.~Eichmair, \textit{Isoperimetric and Weingarten surfaces in the Schwarzschild manifold}, J. Differential Geom. \textbf{94} (2013), no.~3, 387--407. 


\bibitem[BDS24]{BDS24}
G.~Brigati, J.~Dolbeault, and N.~Simonov, \emph{Logarithmic {S}obolev and
  interpolation inequalities on the sphere: Constructive stability results},
  Ann. Inst. H. Poincar{\'e} C Anal. Non Lin{\'e}aire \textbf{41} (2024),
  no.~5, 1289--1321.

\bibitem[BBI01]{BBI01} D.~Burago, Y.~Burago, and S.~Ivanov, \emph{A course in metric geometry}, Graduate Studies in Mathematics \textbf{33}, AMS, Providence RI (2001).

\bibitem[BG25]{BG25} A.~Burtscher and L.~García-Heveling, \textit{Time Functions on Lorentzian Length Spaces}, Ann. Henri Poincaré \textbf{26} (2025), 1533--1572.

\bibitem[CFL14]{CFL14}
E.~A. Carlen, R.~L. Frank, and E.~H. Lieb, \emph{{S}tability estimates for the
  lowest eigenvalue of a {S}chrödinger operator}, Geom. Funct. Anal.
  \textbf{24} (2014), 63--84.

\bibitem[CD+19]{CD+19}
J.~A.~Carrillo, M.~G.~Delgadino, J.~Dolbeault, R.~L.~Frank, and F.~Hoffmann, \emph{Reverse Hardy–Littlewood–Sobolev inequalities}, J. Math. Pures Appl. \textbf{132} (2019), 133--165.

\bibitem[CM22]{CM22} F.~Cavalletti and A.~Mondino, \textit{A review of Lorentzian synthetic theory of timelike Ricci curvature bounds}, Gen. Relativity Gravitation \textbf{54} (2022), no.~137.

\bibitem[CM25]{CM25}  F.~Cavalletti and A.~Mondino, \textit{A sharp isoperimetric-type inequality for Lorentzian spaces satisfying timelike Ricci lower bounds}, Preprint (2025), arXiv:2401.03949v2.


\bibitem[Ch01]{Ch01} I.~Chavel, \textit{Isoperimetric Inequalities}, Cambridge Tracts in Mathematics \textbf{145}, Cambridge University Press, Cambridge (2001).

\bibitem[CESY21]{CESY21}  O.~Chodosh, M.~Eichmair, Y.~Shi, and H.~Yu, \textit{Isoperimetry, scalar curvature, and mass in asymptotically flat Riemannian 3-manifolds}, Comm. Pure Appl. Math. \textbf{74} (2021), no.~4, 865--905.

\bibitem[CL12]{CL12}  M.~Cicalese and G.~P.~Leonardi, \textit{A Selection Principle for the Sharp Quantitative Isoperimetric Inequality}, Arch. Ration. Mech. Anal. \textbf{206} (2012), 617--643.

\bibitem[DE+25]{DE+25}
J.~Dolbeault, M.~J.~Esteban, A.~Figalli, R.~L.~Frank, and M.~Loss, \emph{Sharp stability for {S}obolev and log-{S}obolev inequalities,
  with optimal dimensional dependence}, Camb. J. Math. \textbf{13} (2025),
  no.~2, 359--430.


\bibitem[EM13]{EM13} M.~Eichmair and J.~Metzger, \textit{Unique isoperimetric foliations of asymptotically flat manifolds in all dimensions}, Invent. Math. \textbf{194} (2013), no.~3, 591--630.

\bibitem[EG15]{EG15} L.~C.~Evans and R.~F.~Gariepy, \textit{Measure Theory and Fine Properties of Functions. Revised Edition}, CRC Press, New York (2015).

\bibitem[Fi13]{Fi13} A.~Figalli, \textit{Stability in Geometric and Functional Inequalities}, In: European Congress of Mathematics, Eur. Math. Soc., Z\"urich (2013), 585--599.

\bibitem[FI13]{FI13}
A. Figalli and E. Indrei, \textit{A Sharp Stability Result for the Relative Isoperimetric Inequality Inside Convex Cones}, J. Geom. Anal. \textbf{23} (2013), 938--969. 

\bibitem[FJ17]{FJ17}
A.~Figalli and D.~Jerison, \emph{Quantitative stability for the
  {B}runn-{M}inkowski inequality}, Adv. Math. \textbf{314} (2017), 1--47.

\bibitem[FMM18]{FMM18}
A.~Figalli, F.~Maggi, and C.~Mooney, \emph{The sharp quantitative {E}uclidean
  concentration inequality}, Camb. J. Math. \textbf{6} (2018), no.~1, 59--87.

\bibitem[FMP10]{FMP10}  A.~Figalli, F.~Maggi, and A.~Pratelli, \textit{A mass transportation approach to quantitative
isoperimetric inequalities}, Invent. Math. \textbf{182} (2010), 167--211.

\bibitem[FZ22]{FZ22}
A.~Figalli and Y.~R.-Y. Zhang, \emph{Sharp gradient stability for the {S}obolev
  inequality}, Duke Math. J. \textbf{171} (2022), no.~12, 2407--2459.

\bibitem[FZ23]{FZ23} A.~Figalli and Y.~R.-Y.~Zhang, \textit{Strong stability of convexity with respect to the perimeter}, Preprint (2023), arXiv:2307.01633.

\bibitem[Fr22]{Fr22}
R.~L.~Frank, \emph{Degenerate stability of some {S}obolev inequalities}, Ann.
  Inst. H. Poincar{\'e} C Anal. Non Lin{\'e}aire \textbf{39} (2022), no.~6,
  1459--1484.

\bibitem[Fr24]{Fr24}
 R.~L.~Frank, \emph{The sharp Sobolev inequality and its stability: An introduction}, In: Geometric and Analytic
Aspects of Functional Variational Principles: Cetraro, Italy 2022, Springer, Cham (2024), 1--64.

\bibitem[FKT22]{FKT22}
	R.~L.~Frank, T.~König, and H.~Tang, \emph{Reverse conformally invariant Sobolev inequalities on the sphere}, J. Funct. Anal. \textbf{282} (2022), no.~4, 109339. 

    \bibitem[FP24a]{FP24a}
R.~L.~Frank and J.~W.~Peteranderl, \emph{Degenerate stability of the {C}affarelli--{K}ohn--{N}irenberg inequality along the {F}elli--{S}chneider
  curve}, Calc. Var. Partial Differential Equations \textbf{63} (2024), no.~44.

\bibitem[FP24b]{FP24b} R.~L.~Frank and J.~W.~Peteranderl, \emph{The sharp $\sigma_2$-curvature inequality on the sphere in
  quantitative form}, Preprint (2024), arXiv:2412.12819.

\bibitem[FPR25]{FPR25} R.~L.~Frank, J.~W.~Peteranderl, and L.~Read, \emph{Sharp quantitative integral inequalities for harmonic extensions}, Preprint (2025), arXiv:2508.09940.
  

\bibitem[Fu89]{Fu89}  B.~Fuglede, \textit{Stability in the isoperimetric problem for convex or nearly spherical domains in $R^n$}, Trans. Amer. Math. Soc. \textbf{314} (1989), 619--638.


\bibitem[Fu15]{Fu15} N.~Fusco, \textit{The quantitative isoperimetric inequality and related topics}, Bull. Math. Sci. \textbf{5} (2015), 517--607.

\bibitem[FMP08]{FMP08} N.~Fusco, F.~Maggi, and A.~Pratelli, \textit{The sharp quantitative isoperimetric inequality}, Ann. Math. \textbf{168} (2008), 941--980.

\bibitem[Ga30]{Ga30} G.~Gamow, \textit{Mass defect curve and nuclear constitution}, Proc. R. Soc. Lond. Ser. A \textbf{126} (1930), 632--644.

 \bibitem[GYZ25]{GYZ25}
	R.~Gong, Q.~Yang, and S.~Zhang, \emph{A simple proof of reverse Sobolev inequalities on the sphere and Sobolev trace inequalities on the unit ball}, J. Funct. Anal. \textbf{290} (2026), no.~9, 111380.

 \bibitem[Gr85]{Gr85}   P.~Grisvard, \emph{Elliptic problems in nonsmooth domains}, Monographs and
Studies in Mathematics \textbf{24}, Pitman Advanced Publishing Program, Boston (1985).

\bibitem[GW04]{GW04}
		P.~Guan and G.~Wang, \emph{Geometric inequalities on locally conformally flat manifolds}, Duke Math. J. \textbf{124} (2004), no.~1,
		177--212.

        \bibitem[GLZ25]{GLZ25}A.~Guerra, X.~Lamy, and K.~Zemas, \emph{Sharp quantitative stability of the
  {M}\"obius group among sphere-valued maps in arbitrary dimension}, Trans.
  Amer. Math. Soc. \textbf{378} (2025), 1235--1259.

\bibitem[Ha92]{Ha92} R.~R.~Hall, \textit{A quantitative  isoperimetric inequality in n-dimensional space}, J. Reine Angew. Math. \textbf{428} (1992), 161--176.

\bibitem[HP70]{HP70} S.~W.~Hawking and R.~Penrose, \textit{The singularities of gravitational collapse and cosmology}, Proc. Roy. Soc. Lond. A. \textbf{314} (1970), 529--548.

\bibitem[Hu06]{Hu06} G.~Huisken, \textit{An isoperimetric concept for mass and quasilocal mass}, Oberwolfach Rep. \textbf{3} (2006), 87--88.

\bibitem[KM13]{KM13}
	H.~Knüpfer and C.~Muratov, \emph{On an isoperimetric problem with a competing nonlocal term I. The planar
case}, Commun. Pure Appl. Math. \textbf{66} (2013), 1129--1162.

\bibitem[KM14]{KM14}
	H.~Knüpfer and C.~Muratov, \emph{On an isoperimetric problem with a competing nonlocal term II. The general
case}, Commun. Pure Appl. Math. \textbf{67} (2014), 1974--1994.

\bibitem[Ko25]{Ko25}
	T.~K{\"o}nig, \emph{Stability inequalities with explicit constants for a family of reverse Sobolev inequalities on the sphere}, Preprint (2025), arXiv:2504.19939.

  \bibitem[KP25]{KP25}
	T.~K{\"o}nig and J.~W.~Peteranderl, \emph{An almost-almost-Schur lemma on the 3-sphere}, Preprint (2025), arXiv:2510.25723.  
    
\bibitem[KS18]{KS18} M.~Kunzinger and C.~Sämann, \textit{Lorentzian length spaces}, Ann. Global Anal. Geom. \textbf{54} (2018), 399--447.

\bibitem[LS21]{LS21} B.~Lambert and J.~Scheuer, \textit{Isoperimetric problems for spacelike domains in generalized Robertson-Walker spaces}, J. Evol. Equ. \textbf{21} (2021), 377--389.

\bibitem[LLS21]{LLS21} C.~Lange, A.~Lytchak, and C.~Sämann, \textit{Lorentz meets Lipschitz}, Adv. Theor. Math. Phys. \textbf{25} (2021), no.~8, 2141--2170.

\bibitem[LP26a]{LP26a} C.~Lange and J.~W.~Peteranderl, \textit{Hausdorff-type metric geometry of the space of Cauchy hypersurfaces}, Preprint (2026), arXiv:2604.11783.


\bibitem[Le23]{Le23} P.~Le, \textit{Lorentz polarisation and isoperimetric inequality in Minkowski spacetime}, Preprint (2023), arXiv:2307.03301.


\bibitem[LP90]{LP90} 
P.-L.~Lions and F.~Pacella, \textit{Isoperimetric inequalities for convex cones}, Proc. Amer. Math. Soc. \textbf{109} (1990), 477--485.


\bibitem[Mc34]{Mc34} 
E.~J.~McShane, \textit{Extension of range of functions}, Bull. Amer. Math. Soc. \textbf{40} (1934), no.~12, 837--842.

\bibitem[Mi35]{Mi35}
  E.~A.~Milne, 
\textit{Relativity, Gravitation and World-Structure}, Clarendon Press, 
 Oxford (1935).

\bibitem[Mi19]{Mi19} E.~Minguzzi, \emph{Lorentzian causality theory}, Living Rev. Relativ. \textbf{22} (2019), no.~3.

\bibitem[On83]{On83} B.~O'Neill, \textit{Semi-Riemannian geometry. With applications to relativity}, Pure and Applied Mathematics \textbf{103}, Academic Press, New York (1983).

\bibitem[Os03]{Os03} W.~F.~Osgood, \textit{A Jordan Curve of Positive Area}, Trans. Amer. Math. Soc. \textbf{4} (1903), no.~1, 107--112. 

\bibitem[Os78]{Os78} R.~Osserman, \textit{The isoperimetric inequality}, Bull. Amer. Math. Soc. \textbf{84} (1978), 1182--1238.

\bibitem[TW22]{TW22}  C.-J.~Tsai and K.-H.~Wang, \textit{An isoperimetric-type inequality for spacelike submanifold in the Minkowski space}, Int. Math. Res. Not. IMRN \textbf{2022} (2022), no.~1, 128--139.
\end{thebibliography}
\end{document}